\documentclass{bcp}

\usepackage[active]{srcltx}
\usepackage{latexsym, amssymb}
\usepackage{amsthm, amsxtra}
\usepackage[colorlinks, bookmarks=true]{hyperref}

\newtheorem{theorem}{Theorem}[section]
\newtheorem{lemma}[theorem]{Lemma}
\newtheorem{corollary}[theorem]{Corollary}

\newtheorem{proposition}[theorem]{Proposition}
\theoremstyle{definition}
\newtheorem{definition}[theorem]{Definition}
\newtheorem{example}[theorem]{Example}

\theoremstyle{remark}
\newtheorem{remark}[theorem]{Remark}

\numberwithin{equation}{section}

\newcommand{\R}{\mathbb{R}}

\newcommand{\N}{\mathbb{N}}

\newcommand{\ep}{\epsilon}

\DeclareMathOperator{\diam}{diam}

\DeclareMathOperator{\vecspan}{span}

\renewcommand{\leq}{\leqslant}
\renewcommand{\geq}{\geqslant}\usepackage{amssymb}

\newcommand{\vr}{\epsilon}

\newcommand{\pp}{\mathbf P}

\newcommand{\tlx}{{\tilde{x}_0}}
\newcommand{\tla}{{\tilde{x}^{(\alpha)}}}
\newcommand{\con}{\mathrm{con}}
\newcommand{\lex}{< \! \! \! <}
\newcommand{\dist}{\mathrm{dist}}

\newcommand{\moco}{\boldsymbol{\sigma}}

\def \dim{{\mathrm{dim}} \, }

\renewcommand{\span}{\mathrm{span}}

\def \diam{{\mathrm{diam}} \, }

\def \eqalign#1{\null\,\vcenter{\openup\jot 
   \ialign{\strut\hfil$\displaystyle{##}$&$
      \displaystyle{{}##}$\hfil \crcr#1\crcr}}\,}

\begin{document}

\mathclass{Primary 54E35; Secondary 54E45, 54E50, 05C05,\\ 47H09, 51F99.}
\abbrevauthors{A. G. Aksoy and T. Oikhberg}
\abbrevtitle{Some results on Metric Trees}

\title{Some results on Metric Trees}

\author{Asuman G\"{u}ven Aksoy}
\address{Department of Mathematics,
        Claremont McKenna College\\
        Claremont, CA 91711\\
E-mail: aaksoy@cmc.edu}

\author{Timur Oikhberg}
\address{Department of Mathematics,
         University of California-Irvine\\
         Irvine, CA, 92697 and \\
Department of Mathematics,University of Illinois at Urbana-Champaign\\ Urbana, IL 61801\\
E-mail: toikhber@math.uci.edu}

\maketitlebcp

\begin{abstract} Using isometric embedding of metric trees into Banach spaces, this
paper will investigate  barycenters, type and cotype, and various
measures of compactness of metric trees. A metric tree ($T$, $d$)
is a metric space such that between any two of its points there is
an unique arc that is isometric to an interval in
$\mathbb{R}$. We begin our investigation by examining isometric
embeddings of metric trees into Banach spaces. We then
investigate the possible images  $x_0=\pi ((x_1+\ldots+x_n)/n)$,
where $\pi$ is a contractive retraction from the ambient Banach
space $X$ onto $T$ (such a $\pi$ always exists) in order to
understand the ``metric'' barycenter of a family of points
$ x_1, \ldots ,x_n$ in a tree $T$.
Further, we consider the metric properties
of trees such as their type and cotype. We identify various
measures of compactness of metric trees (their covering numbers,
$\epsilon$-entropy and Kolmogorov widths) and the connections
between them. Additionally, we prove  that the limit of the  sequence
of Kolmogorov widths of a metric tree is equal to its ball
measure of non-compactness.
\end{abstract}

\section{Introduction}\label{intro}

The study of injective envelopes of metric spaces, also known as metric trees,
(T-theory or $\R$-trees) began with J.~Tits \cite{Tits} in 1977 and since then,
applications have been found within many fields of mathematics. For an
overview of geometry, topology, and group theory applications,
consult Bestvina \cite{Best}. For a complete discussion of these spaces
and their relation to global NPC spaces we refer to
\cite{Brid}. Applications of metric trees in biology and medicine
involve phylogenetic trees \cite{Semp}, and in computer science involve
 string matching \cite{Bart}.
Universal properties of ``$\ell_1$ trees'' (this is a special class of separable
metric trees) in the family of separable complete metrically convex metric spaces have been
discovered, and used to investigate Lipschitz quotients of Banach spaces, in \cite{JLPS}.

Since metric trees are described by
three different names and several definitions,
we start with some definition we will use.
The first three are classical, and can be found, for instance,
in \cite{Blum} or \cite{Ev}.

\begin{definition}\label{D:mseg}
    Let $x, y \in M$, where ($M$, $d$) is a metric space.
    A \emph{geodesic segment from $x$ to $y$} (or a \emph{metric segment},
denoted by $[x,y]$) is the image of an isometric embedding
$\alpha : [a,b]\rightarrow M$ such that $\alpha(a)=x$ and $\alpha(b)=y$.
A metric space is called {\it geodesic} if any two points can be connected
by a metric segment.
    \end{definition}

\begin{definition}\label{D:mt2}
     A metric space $(M,d)$, is a {\it metric tree} if and only if for all $x,y,z \in
    M$, the following holds:
    \begin{enumerate}
        \item there exists a unique metric segment from $x$ to $y$,
        and
        \item $[x,z] \cap [z,y] = \{z\} \Rightarrow [x,z] \cup [z,y] = [x,y]$.
    \end{enumerate}
\end{definition}

\begin{definition}(Metric convexity)\label{D:metr_conv}
Suppose $T$ is a metric tree. A subset $S \subset T$ is called
{\it metrically convex} if $S$ contains the metric segment connecting
any two of its points. The {\it metric convex hull} of $S$
(the smallest metric convex set containing $S$) is denoted by $\con S$.
\end{definition}

\begin{definition}\label{D:leaves}
Suppose $T$ is a metric tree. Following \cite{AkBo}, we call a $t \in T$ a {\it leaf}
(or a {\it final point}) of $T$ if $t \in [x,y]$ (with $x,y \in T$)
implies $t \in \{x,y\}$.
\end{definition}

Below are  some examples of metric trees.

\begin{example}(The Real Tree)
Let $X_{\mathbb{R}}$ denote the set of all bounded  subsets of $\mathbb{R}$ which contain their infimum.
For all subsets $x$ and $y$ in $\mathbb{R}$, define a map $ d: X_{\mathbb{R}} \times X_{\mathbb{R}} \rightarrow \mathbb{R}$ by
$$ d(x,y):= 2 \max \{ \sup (x \triangle y), \inf x, \inf y\} -(\inf x +\inf y)$$ where by $x \triangle y$ we mean the symmetric difference of the sets $x$
and $y$. Then $d$ is a metric on $X_{\mathbb{R}}$, and $(X_{\mathbb{R}},d )$
 is a metric tree. For striking properties of this metric tree we refer to \cite{DMT}.
\end{example}



\begin{example}(The Radial Metric, Spider Tree)\label{E:radial}
    Define $d: \R^2 \times \R^2 \to \R_{\geq 0}$ by:
    \[
        d(x,y) =
            \begin{cases}
            \|x-y\| & \text{if $x = \lambda \, y$ for some $\lambda \in \R$,}\\
            \|x\| + \| y \| & \text{otherwise.}
            \end{cases}
    \]
    We can observe the $d$ is in fact a metric and that $(\R^2,d)$ is a metric tree.
\end{example}

\begin{example}(Finitely generated trees)\label{E:graph}
Suppose ${\mathcal{T}}$ 
is a weighted graph-theoretical tree, whose sets of vertices and
edges are denoted by ${\mathcal{V}}$ and ${\mathcal{E}}$, respectively.
Let ${\mathbf{d}}_e$ denote the length (weight) of the edge $e$.
We construct the metric tree $\tilde{\mathcal{T}}$, generated by
${\mathcal{T}}$, as a union of {\it elementary segments} $[v_1,v_2]$,
where $v_1, v_2 \in {\mathcal{V}}$ are adjacent.
In this case, $[v_1,v_2]$ is identified with the edge $e$,
connecting $v_1$ and $v_2$, and $d(v_1,v_2) = {\mathbf{d}}_e$.
The definition of the distance $d$ is then extended to
${\tilde{\mathcal{T}}}^2$ in the obvious way. It is easy to see that
all the tree axioms are satisfied.
\end{example}

\begin{example}(The spider with $n$ legs)\label{E:tripod}
In many of our examples, we shall consider a subtree of the radial
tree described above. Fix $n \in \N$, and a sequence of positive
numbers $(a_i)_{i=1}^n$, the {\it spider with $n$ legs} is defined
as a union of $n$ intervals of lengths $a_1, \ldots, a_n$, emanating
from the common center, and equipped with the radial metric. More precisely,
our tree $T$ consists of its center $o$, and the points $(i,t)$, with
$1 \leq i \leq n$ and $0 < t \leq a_i$. The distance $d$ is defined
by setting $d(o, (i,t)) = t$, and
$$
d((i,t),(j,s)) = \left\{ \begin{array}{ll}
   |t-s|   &   i=j   \cr
   t+s     &   i \neq j
\end{array} \right. .
$$
Abusing the notation slightly, we often identify $o$ with $(i,0)$.

We can consider this tree as a finitely generated tree (see Example~\ref{E:graph}),
arising from a ``spider-like'' graph, with vertices $v_0, v_1 \ldots, v_n$,
and edges of length $1$ connecting $v_0$ with $v_1, \ldots, v_n$.

The simplest spider -- that with three legs -- is called a {\it tripod}
(this terminology comes from \cite{St}).
\end{example}

   \begin{example}(Non-simplicial Tree)
   In general metric trees are more complicated than metric graphs. Metric graphs are spaces obtained by taking connected graphs and metrizing nontrivial edges. Such a graph is a metric tree if the corresponding  metric graph is connected and simply connected. For example, consider the set $[0,\infty) \times [0,\infty)$ with the distance
   $d: [0,\infty)\times [0,\infty) \to \R_{\geq 0}$  given by:
        \[
            d(x,y) =
            \begin{cases}
            |x_1-y_1| & \text{if $x_2 =y_2 $,}\\
            x_1+y_1+ |x_2-y_2| & \text{if $x_2 \neq y_2 $.}

            \end{cases}
           \]
Set $X_n= ( \mathbb{H}^n, \frac{1}{n} \,d)$, where $\mathbb{H}^n$ is a hyperbolic n-space. Then the ultraproduct $\prod X_n$ over some nontrivial ultrafilter $\mathcal{U}$ is the asymptotic  cone $\mathbb{H}_{\mathcal U}^n$ of $\mathbb{H}^n$, an example of a non-simplicial tree. In this metric tree, the complement of every point has infinitely many connected components. For further discussion of this space and construction of metric trees related to the asymptotic geometry of hyperbolic metric spaces we refer to \cite{Brid} and \cite{DP}.

   \end{example}

   We refer the reader to \cite{Blum} for the properties of metric segments and to \cite{AkBo}, \cite{AkKh} and \cite{Ev}
 for the basic properties of complete metric trees.
Below we list some useful notation and results. 

For $x, y$ in a metric space $M$, we sometimes write $xy = d(x,y)$.
For $x, y, z \in M$, we say  $y$ is \emph{between $x$ and $z$}, denoted
    $xyz$, if and only if $xz = xy + yz$.
The following facts will be used throughout the paper:
 \begin{enumerate}
 \item (Transitivity of betweenness \cite{Blum}) Let $M$ be a metric space and let 
$a,b,c,d \in M$. If $abc$ and $acd$, then $abd$ and $bcd$.
 \item  (Three point property \cite{AkBo}, \cite[Section 3.3.1]{Ev})
Let $x,y,z \in T$ ($T$ is a complete metric tree).
There exists (necessarily unique) $w \in T$ such that $$[x,z] \cap [y,z] = [w,z]\,\,\,
       \mbox{and}\,\,\, [x,y] \cap [w,z] = \{w\} $$
       Consequently, $$[x,y] = [x,w] \cup [w,y],\,\,\,[x,z] = [x,w] \cup [w,z],\,\,\,
       \mbox{and}\,\,\, [y,z] = [y,w] \cup [w,z].$$

 \item (Compactness, \cite{AkBo})
  A metric tree $T$ is compact if and only if
$$T = \displaystyle \bigcup_{f \in F} [a,f]\,\,\,\mbox{ for all}\,\,\, a \in T\,\,\mbox{and}\,\, \overline{F} \,\,\mbox{is compact}\,,$$ where  $F$ is the set of leaves of $T$.

 \end{enumerate}

We also need to mention several properties of metric spaces.

\begin{definition}\label{D:4pts}
A metric space $(X,d)$ is said to be {\it $0$-hyperbolic}
(or to satisfy {\it the four-point inequality}) \cite{DMT, Ev}
if, for any $x_1, x_2, x_3$, $x_4$ in $X$,
$$ 
d(x_1,x_2) + d(x_3,x_4) \leq \max\{ d(x_1,x_3) + d(x_2,x_4) ,
d(x_1,x_4) + d(x_2,x_3) \} .
$$ 
$X$ is said to satisfy {\it Reshetnyak's inequality} \cite{St} if, for any
$x_1, x_2, x_3, x_4 \in X$,
$$ 
d(x_1,x_2)^2 + d(x_3,x_4)^2 \leq d(x_1,x_3)^2 + d(x_2,x_4)^2 +
d(x_1,x_4)^2 + d(x_2,x_3)^2 .
$$ 
\end{definition}

It was proven in \cite{Dr} (see also \cite[Chapter 3]{Ev}) that any $0$-hyperbolic metric
space embeds isometrically into a metric tree. Moreover, a metric space $M$
is a metric tree if and only if it is $0$-hyperbolic and geodesic.
We see below (Lemma~\ref{l:4pts}) that the four-point inequality implies
Reshetnyak's inequality. The converse is not true, as an example
of a Hilbert space shows.

Further properties of metric spaces are encoded in the definition below.


\begin{definition}\label{D:NPC}
A geodesic metric space $X$ is called a {\it CAT(0) space} or a global metric space of non-positive curvature (NPC),
({\it global NPC space})
if for every three points $x_0, x_1, x_2 \in X$, the {\it CN Inequality} holds:
$$
d(x_0,y)^2 \leq
\frac{d(x_0,x_1)^2}{2} + \frac{d(x_0,x_2)^2}{2} - \frac{d(x_1,x_2)^2}{4}
$$
whenever $y$ is the midpoint of a metric segment connecting $x_1$ and $x_2$.
\end{definition}

For information on these spaces, the reader is referred to \cite{Brid}, \cite{St},
or \cite{O-P}.
In \cite{sato}, it was shown that a geodesic space is a CAT(0) space if and only if
it satisfies Reshetnyak's inequality.
The class of CAT(0) spaces includes metric trees (see Lemma~\ref{l:4pts}), as well as Hilbert spaces and hyperbolic spaces \cite{GR}.


Generalizing the classical Banach space notion of uniform convexity, we follow
\cite{GR} in defining the modulus of convexity for geodesic metric spaces.

\begin{definition}\label{D:convexity}
Suppose $(M,d)$ is a geodesic metric space. For numbers $R > 0$, $\vr \in [0,2R]$, and
$a \in M$, let $\moco_M(a,R,\vr) = \inf\{1 - d(m,a)/R\}$, where $m$ is the
midpoint of a metric segment connecting $x_1$ and $x_2$, and the infimum runs over
all pairs $(x_1,x_2)$ with $\max\{d(a,x_1), d(a,x_2)\} \leq R$, and
$d(x_1,x_2) \geq R \vr$. Define the {\it modulus of convexity} of $M$
by setting $\moco_M(R,\vr) = \inf_{a \in M} \moco_M(a,R,\vr)$.
\end{definition}

The CN inequality implies that, for any CAT(0) space $M$,
$$
\moco_M(R,\vr) \geq \moco_H(R,\vr) = \moco_H(\vr) = 1 - \sqrt{1 - \vr^2/4}
$$
(here, $H$ is the Hilbert space of dimension greater than $1$).
In Lemma~\ref{l:convex}, we obtain a sharper estimate on the moduli of convexity
of metric trees.

\section{Hyperconvexity and Metric Trees}

\begin{definition}\label{D:hc}
    A metric space $X$ is \emph{hyperconvex} if
    \[
    \bigcap_{i \in I} B_c(x_i;r_i) \, \not= \emptyset
    \]
    for any collection $\{B_c(x_i;r_i)\}_{i \in I}$ of closed balls in $X$ with
    $x_ix_j \leq r_i + r_j$.
\end{definition}

The notion of a hyperconvex metric space was introduced by Aronszajn and Panitchpakdi \cite{Aron}.  They proved the following theorem, which is now well known.

\begin{theorem}[Aronszajn and Panitchpakdi, \cite{Aron}]\label{T:neehc}
    $X$ is a hyperconvex metric space if and only if $X$ is a $1$ absolute Lipschitz retract; that is,
    for all metric spaces $D$, if $C \subset D$ and $f: C \to X$ is a nonexpansive
    mapping, then $f$ can be extended to the nonexpansive mapping
    $\tilde{f}:D \to X$.
\end{theorem}
Hyperconvex spaces are complete and connected; the simplest example of hyperconvex space is the set of real numbers $\mathbb{R}$ or a finite dimensional real Banach space endowed with the maximum norm. While
the Hilbert space $\ell_2$ fails to be hyperconvex, the spaces
$\ell_{\infty}$ and $L_{\infty}$ are hyperconvex.
The connection between hyperconvex metric spaces and metric trees is given in the following theorem:

\begin{theorem}[\cite{KirkH}, \cite{Akso}]\label{T:cmt is hc}
    A complete metric tree $T$ is hyperconvex.
Conversely, any hyperconvex space with unique metric segments
is a complete metric tree.
\end{theorem}



\section{Embeddings of Metric Trees into Banach spaces}\label{s:embed}

Henceforth, we consider {\it isometric} embeddings of metric
trees into Banach spaces. Note that there is a wealth of results concerning
{\it Lipschitz} embeddings of graphs (including trees) into Banach spaces.
In particular, the connections between Lipschitz embeddability of trees and
superreflexivity were investigated in
\cite{Bou}, and more recently, in \cite{Bau} (they are discussed in
more detail in Section~\ref{subs:superrefl}). The distortion necessary to embed
a metric tree into a uniformly convex Banach space can be found in e.g.~\cite{Mat}
(by \cite{LNP}, this problem is equivalent to computing the distortion of
embedding the corresponding finite tree).

\subsection{Embeddings into $L_\infty$}\label{subs:into_l_infty}
First we consider two embeddings into $L_\infty$, with different properties.

\begin{theorem}[From \cite{KirkHB}, page 395]\label{T:ms into L infty}
    Let $X$ be a metric space and $a \in X$, then
$J = J_a : X \to \ell^\infty(X) : x \mapsto (xm - am)_{m\in X}$
    is an isometric embedding of $X$ into $\ell^\infty(X)$.
\end{theorem}



The embedding $J_a$ defined above is called {\it canonical}.
When the space $X$ is bounded, we can also use the embedding $J(x)(y)= d(x,y)$.

We can also embed a metric space $X$ into a larger $L_\infty$ space.
To this end,
pick $t_0 \in X$, and denote by
${\mathcal{L}}_{X,t_0} = {\mathcal{L}}$ the space of
$1$-Lipschitz functions from $X$ to $\R$, vanishing at $t_0$.
Define the {\it universal embedding} of $X$ into
$\ell_\infty({\mathcal{L}})$ by setting, $U(t) = (f(t))_{f \in {\mathcal{L}}}$ for $t \in X$ .
Below we show that $U$ is indeed an isometric embedding,
satisfying a certain ``universal projective'' property.

\begin{theorem}\label{thm:universal}
\begin{enumerate}
\item
The map $U$ described above is an isometry.
\item
For any $1$-Lipschitz function $g : X \to \R$, there exists
a $1$-Lipschitz affine functional
$\tilde{g} : \ell_\infty({\mathcal{L}}) \to \R$, such that
$g = \tilde{g} \circ U$.
\item
For any $1$-Lipschitz function $g : X \to Z$, where $Z$ is
a $\lambda$-injective Banach space,
there exists a $\lambda$-Lipschitz affine map
$\tilde{g} : \ell_\infty({\mathcal{L}}) \to Z$, such that
$g = \tilde{g} \circ U$.
\end{enumerate}
\end{theorem}

\begin{proof}
(1) Fix $x, y \in X$, and show that $\|U(x) - U(y)\| = xy$.
As any $f \in {\mathcal{L}}$ is $1$-Lipschitz, the
definition of $U$ yields
$$
\|U(x) - U(y)\| = \sup_{f \in {\mathcal{L}}} |f(x) - f(y)| \leq xy .
$$
To prove the reverse inequality,
consider the function $f_x : X \to \R : t \mapsto xt - xt_0$.
Clearly, $f_x \in \mathcal{L}$, hence
$\|U(x) - U(y)\| \geq |f_x(x) - f_x(y)| = xy$.
Thus, $U$ is an isometry.

(2)
By translation, we can assume that $g(t_0) = 0$, hence
$g \in \mathcal{L}$.
Set $\tilde{g}(a) = a_g$ for $a = (a_f)_{f \in \mathcal{L}} \in \ell_\infty(\mathcal{L})$,
 Then $\tilde{g}$ is
a contractive linear functional. Moreover,
for any $x \in X$, $\tilde{g}(U(x)) = (U(x))_g = g(x)$,
as desired.

(3)
Fix an isometric embedding $I : Z \to \ell_\infty(\Gamma)$.
Let $P : \ell_\infty(\Gamma) \to Z$ be
a projection of norm not exceeding $\lambda$. We can view $I \circ g$ as a collection of maps $h_\gamma : X \to \R$ ($\gamma \in \Gamma$). By Part (2),
each of them admits a $1$-Lipschitz extension $\tilde{h}_\gamma$.
This results in a $1$-Lipschitz map $\tilde{h} = (\tilde{h}_\gamma) : \ell_\infty(\mathcal{L}) \to \ell_\infty(\Gamma)$, extending $I \circ g$.
We complete the proof by setting $\tilde{g} = P \circ \tilde{h}$.
\end{proof}

Note that the canonical embedding of $T$ into $\ell_\infty(T)$
need not share this property of the universal embedding. Indeed,
suppose $T = [0,1]$. Consider the function $g : T \to \R$,
defined by  setting $g(0) = 0$, $g(1/n) = 0$, $g((2n+1)/(2n(n+1))) =
1/(2n(n+1))$ ($n \in \N$), and letting $g$ be linear on each interval $[ \frac{1}{n+1}, \frac{2n+1}{2n(n+1)}]$ and $[\frac{2n+1}{2n(n+1)}, \frac{1}{n}]$.
  We claim that there is
no affine bounded map $\tilde{g} : \ell_\infty(T) \to \R$ such that
$\tilde{g} \circ J = g$, where $J$ is the canonical embedding.
To show this, recall the definition of $J$ (with $x^* = 0$):
for $x, y \in T$,
$$
h_x(y) = J(x)(y) = |x - y| - |y| =
\left\{ \begin{array}{ll}
   x - 2y   &   y \leq x   \cr
   -x       &   y \geq x
\end{array} \right. .
$$
For $n \in \N$, set $a_n = (1/n+1/(n+1))/2 = (2n+1)/(2n(n+1))$,
$b_n = (1/n-1/(n+1))/2 = 1/(2n(n+1))$,
$F_n = h_{1/n}$, and $G_n = h_{a_n}$. By definition, $g(F_n) = 0$,
and $g(G_n) = b_n$. Furthermore, $h_0 = 0$, and $g(h_0) = 0$.
Therefore, the extension $\tilde{g} : \ell_\infty(T) \to \R$,
if it exists, is a linear functional.

Let $H_n = F_{n+1} + F_n - 2 G_n$. A simple computation shows that
$$
H_n(y) = \left\{ \begin{array}{ll}
   2(1/n - y)       &   a_n \leq y \leq 1/n   \cr
   2(y - 1/(n+1))   &   1/(n+1) \leq y \leq a_n   \cr
   0                &   {\mathrm{otherwise}}
\end{array} \right. ,
$$
hence $\|H_n\| = \sup_y |H_n(y)| = 2 b_n$.
By linearity, $\tilde{g}(H_n) = - 2 b_n$.

For $N \in \N$, let $H = \sum_{n=1}^N (2b_n)^{-1} H_n$. Then $\|H\| = 1$,
and $\tilde{g}(H) = - N$. As $N$ is arbitrary, there
is no $\tilde{g}$ with the desired properties.

Note also that there need not be an ``injective'' counterpart of
the ``projective'' universal embedding $U$. More precisely,
suppose $T$ is the ``tripod'' tree, described in
Example~\ref{E:tripod}.
There is no isometric embedding
$A : T \to X$ with the property that, for any
isometric embedding $B : T \to Y$ ($X$ and $Y$ are
Banach spaces), there exists a contractive affine map
$V : Y \to X$ satisfying $V \circ B = A$. Indeed,
suppose, for the sake of contradiction, that there exists
an $A$ with this property.
Consider $B_1 : T \to \ell_\infty^2$, taking
$(1,t)$ to $t(1,1)$, $(2,t)$ to $t(1,-1)$, and
$(3,t)$ to $-t(1,-1)$. We can assume that $A(o) = 0$
(as before, $o$ denotes the ``root'' of $T$).
Suppose $V_1 \circ B_1 = A$, for some $V_1$. Then
$A(2,1) = - A(3,1)$.
Modifying $B_1$ to obtain the ``right'' $B_2$ and $B_3$,
we show that $A(1,1) = - A(3,1)$,
and $A(1,1) = - A(2,1)$. But these three equalities cannot hold
simultaneously.

\subsection{Embeddings into $L_1$}
Next we define the ``semicanonical'' embedding of $T$ into a space $L_1(\mu)$,
with the measure $\mu$ on $T$ defined below (we follow the construction from
\cite{godard}). For any two points $x, y \in T$, denote by $\phi_{xy}$
the isometry from  $[0,d(x,y)]$ to $[x,y]$, mapping $0$ to $x$.
A set $S \subset T$ is said to be measurable if $\phi_{xy}^{-1}(S)$ is a
Lebesgue measurable subset of $[0,d(x,y)]$ for any $x, y \in T$.
By the transitivity of betweenness (Section~\ref{intro})
the intersection of two metric segments
is either empty, a singleton, or a metric segment, hence any metric
segment is measurable.

For an interval $[x,y] \subset T$ and measurable $S \subset T$, we
define $\mu_{[x,y]}(S) = \lambda(\phi_{xy}^{-1}(S))$, where $\lambda$ is the
Lebesgue measure. Denote now by ${\mathcal{F}}$ the set of all finite unions
$F = \cup_{k=1}^n [x_k, y_k]$ of disjoint unions of metric segments.
For a measurable $S \subset T$, and $F$ as above, set
$\mu_F(S) = \sum_{k=1}^n \mu_{[x_k,y_k]}(S)$. Finally,
let $\mu_T(S) = \mu(S) = \sup_{F \in {\mathcal{F}}} \mu_F(S)$.
It is easy to see that $\mu$ is indeed a measure, vanishing on
countable sets, such that $\mu([x,y]) = d(x,y)$ for any $x, y \in T$.

The ``semicanonical'' embedding $U = U_{x_0}$ of $T$ into $L_1(\mu)$
($x_0$ is a point in $T$), is defined by $U(x) = \chi_{[x_0,x]}$.
To verify that $U$ is isometric, note that, for any $x, y \in T$, there exists a
unique $z \in [x_0,x]$ s.t. $[x,y] = [x,z] \cup [z,y]$, and
$[z,y] \cap [x_0,x] = \{z\}$. Then $U(x) - U(y) = \chi_{[x,z)} - \chi_{[y,z)}$,
hence
$$
\|U(x) - U(y)\| = \|\chi_{[x,z)}\| + \|\chi_{[y,z)}\| =
d(x,z) + d(y,z) = d(x,y) .
$$

An embedding of a finitely generated tree into $\ell_1^N$
is described into \cite[Section 2.5]{Ev}.

\subsection{A characterization of superreflexivity}\label{subs:superrefl}
Recall that a Banach space $X$ is called {\it superreflexive} if all its
ultrapowers are reflexive, or equivalently, any Banach space which can be
finitely represented in $X$ must be reflexive. The reader is referred to
\cite{Beau} for many properties and characterizations of superreflexive
spaces.

\begin{theorem}\label{thm:superrefl}
Suppose $X$ is a non-superreflexive Banach space, $T$ is a finitely
generated metric tree, and $\vr > 0$. Then there exists a
Banach space $Y$, $(1+\vr)$-isomorphic to $X$, such that $T$
embeds into $Y$ isometrically.
\end{theorem}

This theorem should be compared with the characterizations of superreflexive
Banach spaces due to F.~Baudier and J.~Bourgain \cite{Bau, Bou}. Their results
concern the {\it binary tree of height $n$} ${\mathcal{T}}_n = \cup_{j=0}^n \{-1,1\}^j$
($n \geq 0$), and the {\it infinite binary tree}
${\mathcal{T}}_\infty = \cup_{j=0}^\infty \{-1,1\}^j$.
We can view these objects as graphs, where the only
edges are those connecting $([\alpha])$ with $([\alpha], \pm 1)$
($\alpha \in \{-1,1\}^j$). The graph structure induces the
{\it hyperbolic distance} $d$, defined as follows.
For $\alpha = (\alpha_1, \ldots, \alpha_k) \in \{-1,1\}^k$ and
$\beta = (\beta_1, \ldots, \beta_\ell) \in \{-1,1\}^\ell$,
denote by $s = s(\alpha, \beta)$ the smallest integer $j$ for which
$\alpha_{j+1} \neq \beta_{j+1}$ (if $\alpha_1 \neq \beta_1$, or if either $k$
or $\ell$ equals $0$, set $s = 0$). Let $d(\alpha, \beta) = k + \ell - 2s$.

The {\it Lipschitz constant} of an embedding $f : A \to B$ between
metric spaces is defined as
$$
{\mathrm{L}}(f) = \sup_{x \neq y} \frac{d_B(f(x), f(y))}{d_A(x,y)}
\sup_{x \neq y} \frac{d_A(x,y)}{d_B(f(x), f(y))}
$$
(here $d_A$ and $d_B$ are the distances in the spaces $A$ and $B$,
respectively). We say that $A$ has a {\it Lipschitz embedding} into $B$
if there exists an embedding $f : A \to B$ with finite Lipschitz constant.
A family $(A_n)$ is said to have a {\it uniform Lipschitz embedding} into $B$
if there exist embeddings $f_n : A_n \to B$, with
$\sup_n {\mathrm{L}}(f_n) < \infty$.

J.~Bourgain \cite{Bou} proved that a Banach space $X$ is not superreflexive
if and only if the family $({\mathcal{T}}_n)$ has uniform Lipschitz embedding
into $X$. Recently, F.~Baudier \cite{Bau} established that ${\mathcal{T}}_\infty$
Lipschitz embeds into any non-superreflexive space. Together with
Theorem~\ref{thm:superrefl}, these results yield:

\begin{theorem}\label{thm:bou_bau}
For a Banach space $X$, the following are equivalent:
\begin{enumerate}
\item
$X$ is not superreflexive.
\item
There exists a Lipschitz embedding of ${\mathcal{T}}$ into $X$.
\item
There exist Lipschitz embeddings $f_n : {\mathcal{T}} \to X$ ($n \in \N$),
with $\sup_n {\mathrm{L}}(f_n) < \infty$.
\item
Any finitely generated metric tree embeds isometrically into an
isomorphic copy of $X$.
\end{enumerate}
\end{theorem}

\begin{remark}\label{r:distort}
It is easy to note that a strictly convex space cannot contain a tripod,
described in Example ~\ref{E:tripod} (a Banach space $X$ is called
strictly convex if the equality $2(\|x\|^2 + \|y\|^2) = \|x-y\|^2$ implies $x=-y$).
If $X$ is separable, we can find an injection $T : X \to \ell_2$, and equip $X$
with the equivalent strictly convex norm $|||x||| = (\|x\|^2 + \|Tx\|^2)^{1/2}$.
Thus, there exist non-superreflexive Banach spaces which do not contain the tripod
isometrically. This shows that renorming is essential in Theorems \ref{thm:superrefl}
and \ref{thm:bou_bau}(4).
\end{remark}

For the proof of Theorem~\ref{thm:superrefl}, we need a simple lemma.

\begin{lemma}\label{l:separate}
Suppose a Banach space $X$ is not superreflexive, and $c \in (0,1)$.
Then for every $n \in \N$ there exists a family $(x_i)_{i=1}^n$ in the
unit ball of $X$, such that
\begin{enumerate}
\item
If $(a_i)$ is a sequence of scalars, changing signs at most once, then\\
$\|\sum_{i=1}^n a_i x_i\| \geq 2 c \sum_{i=1}^n |a_i|$.
\item
For every $i$, $\dist(x_i, \span[x_j : j \neq i]) > c$.
\end{enumerate}
\end{lemma}

\begin{proof}
Fix $\lambda \in (c,1)$. By a Ramsey-style result from \cite{HKPTZ},
there exists $m = m(n, c, \lambda) \in \N$ with the following property:
if $(y_i)_{i=1}^m$ is a subset of the unit ball of a Banach space $X$,
and $\|y_i - y_j\| \geq 2 \lambda$ whenever $i \neq j$, then there exist
$1 \leq s_1 < s_2 < \ldots < s_n \leq m$, such that
$dist(y_{s_i}, \span[y_{s_j} : j \in \{1, \ldots, n\} \backslash \{i\}]) > c$
for any $i \in \{1, \ldots, n\}$.

Now suppose $X$ is not superreflexive. By \cite{SS} (see also \cite[Part 4]{Beau}),
there exist $y_1, \ldots, y_m$ in the unit ball of $X$, such that, for every
$1 \leq k \leq m$, $\|y_1 + \ldots + y_k - y_{k+1} - \ldots - y_m\| > m + \lambda - 1$.
We claim that $\|\sum_{j=1}^m \alpha_j y_j\| \geq \lambda \sum_{j=1}^m |\alpha_j|$
if the sequence $(\alpha_j)$ changes sign at most once. Indeed, suppose
$\alpha_j \geq 0$ for $j \leq k$, and $\alpha_j \leq 0$ for $j \geq k+1$.
By scaling, we can assume $\sum_j |\alpha_j| = 1$. Then
$0 \leq \alpha_j \leq 1$ for $1 \leq j \leq k$, and
$0 \geq \alpha_j \geq -1$ for $k+1 \leq j \leq m$.
By the triangle inequality,
$$
\|\sum_{j=1}^m \alpha_j y_j\| \geq \|y_1 + \ldots + y_k - y_{k+1} - \ldots - y_m\| -
\sum_{j=1}^m (1 - |\alpha_j|) > (m + \lambda - 1) - m + 1 = \lambda .
$$
By our choice of $m$, we can find the vectors $x_k = y_{s_k}$ with the
required properties.
\end{proof}

\begin{proof}[Proof of Theorem~\ref{thm:superrefl}]
A weighted tree graph ${\mathcal{T}} = ({\mathcal{V}}, {\mathcal{E}})$
(${\mathcal{V}}$ and ${\mathcal{E}}$ denote the sets of vertices and edges,
respectively) gives rise to the metric tree $T$, as in Example~\ref{E:graph}.
Select $v_\emptyset \in {\mathcal{V}}$, and call it the {\it root}.
Enumerate the immediate descendants of $v_\emptyset$
(that is, the vertices connected to $v_\emptyset$ by edges)
by $v_1, \ldots, v_{n_\emptyset}$.
For $1 \leq i \leq n_\emptyset$, let $a_i = d(v_\emptyset, v_i)$.
For each $i$, enumerate its own immediate descendants
$v_{11}, \ldots, v_{1n_1}$, and set $a_{ij} = d(v_i, v_{ij})$.
Proceeding further in the same manner, we write ${\mathcal{V}}$ as
the collection of points $v_S$, for a finite collection ${\mathcal{S}}$
of finite strings $S$.
Then $v_{S^\prime}$ is a descendant of $v_S$ if and only if $S^\prime = S \smile j$,
for some $j \in [1, n_S]$. Set $a_{S^\prime} = d(v_S, v_{S^\prime})$,
where $S$ is the unique immediate predecessor of $S^\prime$.

For $S_1, S_2 \in {\mathcal{S}}$, write $S_1 \prec S_2$ if $v_{S_1}$ is a
predecessor of $v_{S_2}$, or equivalently, if
$S_2 = S_1 \smile j_1 \smile \ldots \smile j_k$.
For $S_\alpha = (i_{1\alpha} \ldots i_{k_\alpha \alpha})$
($\alpha \in \{1,2\}$), set $S_1 \wedge S_2 = (i_{11} \ldots i_{k_01})$,
where $k_0$ is the largest integer $k$ with the property that
$i_{k1} = i_{k2}$. If there is no such $k$, set $S_1 \wedge S_2 = \emptyset$.
Then $v_{S_1 \wedge S_2}$ is the largest common predecessor of $v_{S_1}$ and
$v_{S_2}$.

It is easy to see that the distance $d$ on the set ${\mathcal{V}}$
(inherited from the tree $T$) is given by the formula described below.
For $S_\alpha = (i_{1\alpha} \ldots i_{k_\alpha \alpha})$
($\alpha \in \{1,2\}$), let $k_0$ be the largest integer $k$ such that
$i_{k1} = i_{k2}$. Let $S = S_1 \wedge S_2 = (i_{11} \ldots i_{k_01})$.
Then
\begin{equation}
d(v_{S_1}, v_{S_2}) = d(v_{S_1}, v_S) + d(v_S, v_{S_2}) =
\sum_{m = k_0 + 1}^{k_1} a_{(i_{11} \ldots a_{m1})} +
\sum_{m = k_0 + 1}^{k_2} a_{(i_{12} \ldots a_{m2})} .
\label{eq:dist_vert}
\end{equation}

The main step is to renorm $X$ (making it into $Y$)
in such a way that there exists an isometry
$J_{\mathcal{V}} : {\mathcal{V}} \to Y$.
We then extend it to $J : T \to Y$ so that
$J|_{\mathcal{V}} = J_{\mathcal{V}}$.
For $t \in T \backslash \mathcal{E}$, there exist
unique $v_1, v_2 \in {\mathcal{V}}$ such that $t$ belongs to the
elementary segment $[v_1, v_2]$.
Let $\lambda = d(v_1, t)/d(v_1, v_2)$.
Define $J(t) = \lambda J_{\mathcal{V}}(v_1) +
(1-\lambda) J_{\mathcal{V}}(v_2)$.
Clearly, $J$ is an isometry on any elementary segment.
By the description of metric segments given in
\cite[Lemma 15.1]{Blum}, $J$ is an isometry on $T$.

To construct $J_{\mathcal{V}} : {\mathcal{V}} \to Y$,
denote the cardinality of ${\mathcal{V}}$ by $N$.
By Lemma~\ref{l:separate}, there exists, for every $M \in \N$, a family
$(x_{iM})_{i=1}^N \subset B(0;1)$ such that
\begin{equation}
\|\sum_{i=1}^N \alpha_i x_{iM}\| \geq (1 + 2^{-M})^{-1} \sum_{i=1}^N |\alpha_i|
\label{eq:change_sign}
\end{equation}
for any sequence $(\alpha_i)$ with at most one change of signs, and
$\|\sum_{i=1}^N \alpha_i x_{iM}\| \geq (1 + 2^{-M})^{-1} \max_i |\alpha_i|$
for any sequence $(\alpha_i)$.

Introduce the lexicographic order $\lex$ on ${\mathcal{S}}$ as follows:
if $S_1 \prec S_2$, then $S_1 \lex S_2$. Otherwise, let
$S = S_1 \wedge S_2$, and write
$S_\alpha = S \smile j_{1\alpha} \smile \ldots \smile j_{m_\alpha \alpha}$
($\alpha = 1, 2$). We say $S_1 \lex S_2$ if $j_{11} \leq j_{12}$.
Let $\phi : {\mathcal{S}} \to \{1, \ldots, N\}$ be the monotone increasing bijection
with respect to the lexicographic order. Define
$J_M : {\mathcal{V}} \to X$ by setting, for $S = (i_1, \ldots i_k)$,
$$
J_M(v_S) =
\sum_{j=1}^k d(v_{i_1 \ldots v_{j-1}}, v_{i_1 \ldots v_j})
x_{\phi(i_1 \ldots i_j) M} ,
$$
and $J_M(v_\emptyset) = 0$,
By \eqref{eq:dist_vert} and \eqref{eq:change_sign},
$$
d(v_{S_1}, v_{S_2}) \geq \|J_M(v_{S_1}) - J_M(v_{S_2})\| \geq
(1 + 2^{-M})^{-1} d(v_{S_1}, v_{S_2}) .
$$

Passing to a subsequence if necessary, we can assume that
$$\Phi(\alpha_1, \ldots, \alpha_N) = \lim_M
\|\sum_i \alpha_i x_{iM}\|$$ exists for every sequence
$(\alpha_i)_{i=1}^N$. For any such sequence,
$\max_i |\alpha_i| \leq \Phi(\alpha_1, \ldots, \alpha_N) \leq \sum_i |\alpha_i|$.
Thus, we can define a normed space $Z$ by setting
$\|\sum_i \alpha_i e_i\| = \Phi(\alpha_1, \ldots, \alpha_N)$, where
$(e_i)_{i=1}^N$ is the canonical basis for $\R^N$.
Denote the span of $(x_{iM})_{i=1}^N$ in $X$ by $Z_M$, and
define the map $U_M : Z_M \to Z : x_{iM} \mapsto e_i$.
Find $M$ so large that $c = \max\{\|U_M\|, \|U_M^{-1}\|\} < \sqrt{1 + \vr}$.
Consider $J = U_M \circ J_M : {\mathcal{V}} \to Z$.
As $\|\sum_i \alpha_i e_i\| = \sum_i |\alpha_i|$ if
the sequence $(\alpha_i)$ changes sign no more than once,
the map $J$ is an isometry.

To renorm $X$, embed $Z$ isometrically
into $\ell_\infty$. Then there exists $\tilde{U} : X \to \ell_\infty$,
with $\tilde{U}|_{Z_M} = U_M$, and $\|\tilde{U}\| \leq c$.
For $x \in X$, define $\|x\|_Y = \max\{c^{-1}\|x\|, \|\tilde{U} x\|\}$.
Clearly, $c^{-1}\|x\| \leq \|x\|_Y \leq c \|x\|$.
Moreover, for $x \in Z_M$, $\|\tilde{U} x\| = \|U_M x\| \geq c^{-1} \|x\|$,
hence $\|x\|_Y = \|U_M x\|$.
In other words, $Y$ contains $Z$ isometrically.

Consider $J = U_M \circ J_M : {\mathcal{V}} \to Z$.
As $\|\sum_i \alpha_i e_i\| = \sum_i |\alpha_i|$ if
the sequence $(\alpha_i)$ changes sign no more than once,
the map $J$ is an isometry. Therefore, the map $J_M : {\mathcal{V}} \to Y$
is an isometry.
\end{proof}


\section{Barycenters of trees}\label{bary}

There have been numerous attempts to find an appropriate ``non-linear''
notion of the barycenter of a set (or of a measure) in a metric space.
Several possible definitions are discussed in \cite{St}.
In this section, we approach this problem for metric trees, using
their injectivity. More precisely:
suppose $U$ is an isometric embedding of a metric tree $T$
into a Banach space $X$, equipped with the norm $\| \cdot \|$.
Suppose $x_1, \ldots, x_n \in T$, and let $\tlx = (x_1 + \ldots + x_n)/n$
be their barycenter in $X$ (we identify $x \in T$ with $U(x) \in X$). Let
$\pp = \pp_{U,T,X}$ be the set of contractive retractions $\pi$ from $X$ onto $U(T)$
(it is non-empty since $T$ is injective). We try to describe $\pp(\tlx)$. More generally, suppose
$\alpha = (\alpha_i)_{i=1}^n$ is a sequence of positive numbers, with
$\sum_{k=1}^n \alpha_k = 1$. Set $\tla = \sum_{k=1}^n \alpha_k x_k$,
and try to describe $\pp(\tla)$.

\begin{proposition}\label{prop:dist}
Suppose $T$ is a complete metric tree, embedded isometrically into a normed space $X$.
For $x_0 \in T$ and $\tilde{x} \in X$, the following are equivalent:
\begin{enumerate}
\item
$x_0 \in \pp(\tilde{x})$.
\item
For any $x \in T$, $d(x_0,x) \leq \|\tilde{x} - x\|$.
\end{enumerate}
If, in addition, $T$ is compact, then the two statements above are equivalent to:
\begin{enumerate}
\setcounter{enumi}{2}
\item
For any leaf (final point) $y \in T$, $d(x_0,y) \leq \|\tilde{x} - y\|$.
\end{enumerate}
\end{proposition}

In the proofs below, we sometimes identify $T$ with its image in the
ambient Banach space, and $d(\cdot, \cdot)$ with $\|\cdot - \cdot\|$.

\begin{proof}
By the injectivity of $T$, (1) holds if and only if there exists a contractive map
$$\pi : T \cup \{\tilde{x}\} \to T\,\,\,\mbox{such that}\,\,\, \pi|_T = I_T,\,\,\,\mbox{and}\,\,\, \pi(\tilde{x}) = x_0.$$
This, in turn, is equivalent to (2). Clearly, (2) implies (3).
To show that, for a compact $T$, the converse is true, recall the
``Krein-Milman Theorem for metric trees'' (Statement (4) in
Section 1, proved in \cite{AkBo}), which asserts that 
$T = \displaystyle \bigcup_{y \in {\mathcal{L}}} [x_0, y]$, where ${\mathcal{L}}$ is the set of
leaves of $T$). 
For $x \in T$, find $y \in {\mathcal{L}}$ such that $x \in [x_0, y]$.
If $\|x_0 - y\| \leq \|\tilde{x} - y\|$, then
$$
\|x_0 - x\| = \|x_0 - y\| - \|x - y\| \leq \|\tilde{x} - y\| - \|x - y\|
\leq \|\tilde{x} - x\| ,
$$
thus $(3)$ implies $(2)$.
\end{proof}

\begin{corollary}\label{cor:dist_tree}
If then $x_0 \in \pp(\tla)$, then
$d(x_0,x) \leq \sum_k \alpha_k d(x_k,x)$ for any $x \in T$.
\end{corollary}

\begin{proof}
By the Proposition~\ref{prop:dist}(2)
$$
\|x_0 - x\| \leq \big\| \sum_k \alpha_k x_k - x \big\| =
\big\| \sum_k \alpha_k (x_k - x) \big\| \leq \sum_k \alpha_k \|x_k - x\|
$$
for any $x \in T$, whenever $\pi(\tla) = x_0$.
\end{proof}

In certain cases, the converse to this corollary is also true:
this is shown by the following two theorems. However, in general,
the converse implication does not hold (Example~\ref{l1_only}).

\begin{theorem}\label{thm:l_infty}
Suppose $T$ is a complete metric tree, embedded into $\ell_\infty(T)$
in the canonical way. For $x_0 \in T$, the following are equivalent:
\begin{enumerate}
\item
$x_0 \in \pp(\tla)$.
\item
$d(x_0,x) \leq \sum_k \alpha_k d(x_k,x)$ for any $x \in T$.
\end{enumerate}
\end{theorem}

\begin{proof}
The implication (1) $\Rightarrow$ (2) follows from
Corollary~\ref{cor:dist_tree}. We establish the converse.
Recall that the canonical embedding takes $x \in T$ to
$h(x) \in \ell_\infty(T)$, where $h(x)(y) = d(x,y) - d(x^*,y)$.
Suppose $d(x_0,x) \leq \sum_k \alpha_k d(x_k,x)$ for any $x \in T$.
By Proposition ~\ref{prop:dist}, we have to show that
$d(x_0,x) \leq \|\tla - h(x)\|$ for any $x \in T$.
We identify $\tla$ with the function $\phi : T \to \R$,
defined by
$$
\phi(y) = \sum_k \alpha_k h(x_k)(y) = \sum_k \alpha_k \|x_k - y\| - \|x^* - y\| .
$$
Then
$$
\eqalign{
\|\tla - h(x)\|
&
=
\sup_{y \in T} |\phi(y) - h(x)(y)| =
\sup_{y \in T} |\sum_k \alpha_k (\|x_k - y\| - \|x - y\|)|
\cr
&
\geq
|\sum_k \alpha_k (\|x_k - x\| - \|x - x\|)| = \sum_k \alpha_k \|x_k - x\| \geq
\|x_0 - x\| .
}
$$
\end{proof}

\begin{theorem}\label{thm:l_1}
Suppose $T$ is a compact metric tree, embedded into
$L_1(\mu_T)$ 
in the semicanonical way. For $x_0 \in T$, the following are equivalent:
\begin{enumerate}
\item
$x_0 \in \pp(\tla)$.
\item
$d(x_0,x) \leq \sum_k \alpha_k \|x_k - x\|$ for any $x \in T$.
\end{enumerate}
\end{theorem}
\begin{proof}
As in Theorem~\ref{thm:l_infty}, we only need to establish (2) $\Rightarrow$ (1).
Suppose $x$ is a leaf of the tree $T$. By Proposition~\ref{prop:dist},
we have to show that, if $x_0 \in T$ is such that
$d(x_0,x) \leq \sum_k \alpha_k d(x_k,x)$, then 
$\|x_0 - x\| \leq \|\tla - x\|$. 
The semicanonical embedding of $T$ into $L_1$ identifies
$t \in T$ with $\chi_{[x_0,t]}$ (by translation, we can identify $x_0$ with $0$).
For each $k$, find $u_k \in T$ satisfying $[x_0,x_k] \cap [x_0,x] = [x_0,u_k]$.
Then $y_k = x_k - u_k = \chi_{[x_k,u_k]}$ and $z_k = x - u_k = \chi_{[x,u_k]}$
have disjoint support, hence 
\begin{equation}
\|x_k - x\| = \|y_k - z_k\| = \|y_k\| + \|z_k\|
\label{sum}
\end{equation}
Furthermore, $\|x_0 - x\| = \|u_k\| + \|z_k\|$.
Thus,
$$
\sum_k \alpha_k \|x_k - x\| = \sum_k \alpha_k (\|y_k\| + \|z_k\|) \geq
\sum_k \alpha_k \|x_0 - x\| = \sum_k \alpha_k (\|u_k\| + \|z_k\|) ,
$$
which is equivalent to
\begin{equation}
\sum_k \alpha_k \|y_k\| \geq \sum_k \alpha_k \|u_k\| .
\label{u_y}
\end{equation}
We have to show that
$$
\|\tla - x\| = \|\sum_k \alpha_k (x_k - x)\| \geq
\sum_k \alpha_k (\|u_k\| + \|z_k\|) = \|x\| .
$$
In view of (\ref{u_y}) and (\ref{sum}), it is enough to prove
that, for any leaf $x \in T$,
$\|\sum_k \alpha_k (x_k - x)\| = \sum_k \alpha_k \|x_k - x\|$.
Thus, it suffices to establish that, at any point $y \in T$,
the signs of $(x_k - x)(y) = \chi_{[x_0,x_k]}(y) - \chi_{[x_0,x]}(y)$
are independent of $k$. If $y \notin [x_0,x]$,
then $\chi_{[x_0,x_k]}(y) - \chi_{[x_0,x]}(y) \geq 0$ for any $k$.
On the other hand, if $y \notin [x_0,x]$, then
$\chi_{[x_0,x_k]}(y) - \chi_{[x_0,x]}(y) \leq 0$ for any $k$.
\end{proof}
\begin{remark}
The sets
$\{x_0 \in T : d(x_0,x) \leq \sum_k \alpha_k d(x_k,x) \, \forall \, x \in T\}$
were briefly discussed in
Remark 7.2(iii) of \cite{St}. Namely, consider the probability
measure $q = \sum_k \alpha_k \delta_{x_k}$. The set of points
described above was denoted by $C^*(q)$.
\end{remark}

As shown by the following example, this set need not be contained
in the metric or linear convex hull of $x_1, \ldots, x_n$
(see Definition~\ref{D:metr_conv} for the definition of metric convexity).

\begin{example}\label{big_set}
As an example, consider the points $x_i = (i,1)$ ($1 \leq i \leq 3$)
in a spider with four legs (defined in Example~\ref{E:tripod}).
If $T$ is embedded into $\ell_\infty(T)$ (respectively $L_1$)
in the canonical (respectively ~semicanonical)
way, then $\pp(\tlx)$ consists of $o$, as well as of
all $(j,t)$ with $1 \leq j \leq 4$ and $t \leq 1/3$.
In particular, $(4,1/3)$ or rather, its canonical or semicanonical image
belongs to neither the metric nor linear convex hull of $\{x_1, x_2, x_3\}$.
\end{example}

Certain information about $\pp(\tlx)$ may be extracted from the
following results.

\begin{proposition}\label{prop:convex}
Suppose a complete metric tree $T$ is embedded isometrically into a normed space $X$,
and $\tilde{x}$ is a point of $X$. Then $\pp(\tilde{x})$ is a closed, metrically convex
subset of $T$.
\end{proposition}

\begin{proof}
Proposition~\ref{prop:dist} implies that $x_0 \in \pp(\tilde{x})$ if and only if
$d(x,x_0) \leq \|x - \tilde{x}\|$ for any $x \in T$. This implies that
$\pp(\tilde{x})$ is closed. Furthermore, suppose $x_1, x_2, \in \pp(\tilde{x})$,
and $x_0 \in [x_1,x_2]$. Then, by Section 2 of \cite{St},
$$
d(x,x_0) \leq \max\{d(x,x_1) , d(x,x_2)\} \leq \|x - \tilde{x}\|
$$
for any $x \in T$, which implies $x \in \pp(\tilde{x})$.
\end{proof}

In certain cases, when the structure of
$x_1, \ldots, x_n$ in $T$ is simple, we can
describe $\pp(\tlx)$ explicitly. For instance, if $n=2$, then $\pp(\tlx) = \{x_0\}$,
where $x_0 \in [x_1,x_2]$ satisfies $d(x_1,x_0) = d(x_1,x_2)/2$ (equivalently,
$d(x_2,x_0) = d(x_1,x_2)/2$). Indeed,
$$
d(x_0,x_1) = \|(x_1 + x_2)/2 - x_1\| = d(x_1,x_2)/2 ,
$$
and similarly, $d(x_2,x_0) = d(x_1,x_2)/2$. If $x_0 \notin [x_1,x_2]$,
then there exists $y \in [x_1,x_2]$ such that $[x_0,x_s] = [x_0,y] \cup [y,x_s]$
for $s=1,2$. Then $d(x_0,x_1) + d(x_0, x_2) > d(x_1, x_2)$, which contradicts
the contractiveness of the map taking $\tlx$ to $x_0$. Thus, $x_0$ is the unique
point of $[x_1,x_2]$ satisfying $d(x_1,x_0) = d(x_1,x_2)/2$.

In a more complex situation, consider the ``tripod'' $T$,
with limbs of length $1$ (described in Example~\ref{E:tripod}).
For $0 \leq \alpha \leq \beta \leq 1$, we define
$(i,[\alpha,\beta]) = \{(i,t) : \alpha \leq t \leq \beta\}$.

\begin{theorem}\label{thm:tripod}
Consider the points $x_i = (i,1)$ ($i = 1,2,3$) in the tripod $T$
described above.
Suppose $S$ is a subset of $T$. Then there exists an embedding of $T$ into a
Banach space $X$ such that $S = \pp(\tlx)$ if and only if there
exist $i_0 \in \{1,2,3\}$ and $0 \leq \alpha \leq \beta \leq 1/3$, such that
either (i) $S = (i_0,[\alpha,\beta])$, or
(ii) $S = (i_0,[0,\beta]) \cup \big(\cup_{i \neq i_0} (i,[0,\alpha])\big)$.
\end{theorem}

\begin{proof}
First suppose $T$ is embedded in a normed space $X$, and show that $\pp(\tlx)$
is of the form described in the theorem. For $1 \leq i \leq 3$, let
$d_i = \|x_i - \tlx\|$. By relabeling, we can assume that $d_1 \leq d_2 \leq d_3$.
Note that $d_3 \leq 4/3$. Indeed,
$$
\eqalign{
d_3
&
=
\Big\|x_3 - \frac{x_1 + x_2 + x_3}{3}\Big\|
\cr
&
=
\frac{1}{3}\|(x_3 - x_1) + (x_3 - x_2)\| \leq
\frac{1}{3}\big( d(x_3,x_1) + d(x_3,x_2) \big) = \frac{4}{3} .
}
$$
Furthermore, $d_1 + d_2 \geq d(x_1,x_2) = 2$, hence in particular,
$d_1 \geq 2/3$, and $d_2 \geq 1$. Let $\beta = d_2 - 1$, and
$\alpha = |d_1 - 1|$.

By Proposition~\ref{prop:dist}, $x_0 = (i,t) \in T$ belongs to $\pp(\tlx)$ if and only if
$d(x_i, x_0) \leq d_i$ for each $i$. Thus, $x_0 = (1,t) \in \pp(\tlx)$ if and only if
two conditions are satisfied:
\begin{enumerate}
\item
$d(x_1, x_0) = 1 - t \leq d_1$, or in other words, $t \geq 1 - d_1$,
which translates to $t \geq \alpha$ or $t \geq 0$, depending on whether
$1 - d_1$ is positive or negative.
\item
$d(x_2, x_0) = 1 + t \leq d_2$, or in other words, $t \leq d_2 - 1 = \beta$.
\end{enumerate}
The third condition, $d(x_3, x_0) = 1 + t \leq d_3$, is subsumed in the
second one, as $d_3 \geq d_2$. Thus, $(1,t) \in \pp(\tlx)$ if and only if
$1 - \beta \leq t \leq 1 - \alpha$.

Similarly, $x_0 = (2,t) \in \pp(\tlx)$ if and only if two conditions are satisfied:
\begin{enumerate}
\item
$d(x_1, x_0) = 1 + t \leq d_1$, or in other words, $t \leq d_1 - 1$,
which means either $t \leq \alpha$ ($d_1 \geq 1$), or there are no
suitable $t$'s ($d_1 < 1$).
\item
$d(x_2, x_0) = 1 - t \leq d_2$, which is always true, since $d_2 \geq 1$.
\end{enumerate}
Thus, the set of $t$ for which $(2,t) \in \pp(\tlx)$ is either
$[0,\alpha]$, or $\emptyset$. The set $\{t : (3,t) \in \pp(\tlx)\}$
is described in a similar fashion.

Next we construct an embedding of $T$ into a Banach space $X$,
for which $\pp(\tlx) = S$.
Suppose first $0 \leq \alpha \leq \beta < 1/3$, and construct an embedding of $T$
into $L_1(0,2)$ with the property that $\pp(\tlx) = (1,[\alpha,\beta])$.
Let $c = 1 - 3 \beta$, and $a = (1-3\alpha)/c$. Define the functions
$f_1, f_2, f_3$ as follows:
$$
\eqalign{
&
f_1(u) = \left\{ \begin{array}{ll}
   1   &   0 \leq u \leq 1   \cr
   0   &   1 < u \leq 2
\end{array} \right. , \, \, \,
f_2(u) = \left\{ \begin{array}{ll}
   -a       &   0 \leq u \leq c/2   \cr
   0        &   c/2 < u \leq 1   \cr
   1-ac/2   &   1 < u \leq 2
\end{array} \right. ,
\cr
&
f_3(u) = \left\{ \begin{array}{ll}
   0           &   0 \leq u \leq 1-c/2   \cr
   -a          &   1-c/2 < u \leq 1   \cr
   -(1-ac/2)   &   1 < u \leq 2
\end{array} \right. .
}
$$
Note that $f_i f_j \leq 0$ for $i \neq j$, hence $\|t f_i - s f_j\| = t + s$
for positive $t$ and $s$. Therefore, the mapping $(i,t) \mapsto t f_i$
describes an embedding of $T$ into $L_1(0,2)$.

The barycenter $\tlx$ corresponds to the function $g$, given by
$$
g(u) = \left\{ \begin{array}{ll}
   -(a-1)/3   &   u \in [0,c/2] \cup (1-c/2,1]   \cr
   1/3        &   c/2 < u \leq 1-c/2   \cr
   0          &   1 < u \leq 2
\end{array} \right. .
$$
Then
$$
\|f_1-g\| = \Big(1 + \frac{a-1}{3}\Big)c + \frac{2}{3}(1-c) =
1 - \frac{1-ac}{3} = 1 - \alpha ,
$$
and
$$
\eqalign{
\|f_2-g\| = \|f_3-g\|
&
=
\frac{c}{2}\Big(a - \frac{a-1}{3}\Big) + \frac{1}{3}(1-c) +
\frac{c}{2} \cdot \frac{a-1}{3} + \Big(1 - \frac{ac}{2} \Big)
\cr
&
= 1 + \frac{1-c}{3} = 1 + \beta .
}
$$
By Proposition~\ref{prop:dist}(3), $\pp(\tlx)$ consists of all points
$x_0 \in T$ such that $d(x_i,x_0) \leq \|x_i - \tlx\|$ for $i \in \{1,2,3\}$;
that is, of all the points $(1,t)$ with $\alpha \leq t \leq \beta$.

To obtain
$S = (i_0,[0,\beta]) \cup \big(\cup_{i \neq i_0} (i,[0,\alpha])\big)$
we modify the above construction, by setting $c = 1 - 3 \beta$, and
$a = (1+3\alpha)/c$. Then $\|f_2 - g\| = \|f_3 - g\| = 1 + \beta$, and
$\|f_1 - g\| = 1 + (ac-1)/3 = 1 + \alpha$.

A modification of this construction works in the ``limit'' case of $\beta = 1/3$.
In this case embed $T$ into $M([0,2])$ (the space of regular Radon measures on $[0,2]$).
As before, let $\mu_1 = f_1 = \chi_{(0,1)}$, and set
$$
\mu_2 = a \delta_0 + (1-a) \delta_2 , \, \, \,
\mu_2 = a \delta_1 - (1-a) \delta_2 ,
$$
with $a \in [0,1]$ to be determined later
(here, $\delta_x$ is the Dirac measure supported by $x$). Once again, it
is easy to check that the map $(i,t) \mapsto t \mu_i$ defines an
embedding of $T$ to $M([0,2])$. The barycenter $\tlx$ corresponds to
the measure
$$
\nu = \frac{1}{3} \Big(a(\delta_0 + \delta_1) + \chi_{(0,1)} \Big)
$$
hence
$$
\|\mu_1 - \nu\| = \frac{2}{3} \big(1+a\big) = 1 - \frac{1-2a}{3} ,
\, \, \, {\mathrm{and}} \, \, \,
\|\mu_2 - \nu\| = \|\mu_3 - \nu\| = \frac{4}{3} .
$$
To obtain $\pp(\tlx) = (1,[\alpha,1/3])$, set $a=(1-3\alpha)/2$
(then $(1-2a)/3 = \alpha$). To end up with
$S = (1,[0,1/3]) \cup \big(\cup_{i \neq 1} (i,[0,\alpha])\big)$,
set $a=(3\alpha+1)/2$.
\end{proof}

\begin{proposition}\label{prop:nearest}
Suppose a 
metric tree $T$ is embedded isometrically into $L_1(\mu)$,
$x_1, \ldots$, $x_n$ are points of $T$, and $\alpha_1  \ldots \alpha_n$ are positive
numbers, satisfying $\sum_k \alpha_k = 1$. If $x_0 \in T$ belongs to $\pp(\tla)$,
then the unique point of $\con(x_1, \ldots, x_n)$, nearest to $x_0$  also belongs to
$\pp(\tla)$.
\end{proposition}

\begin{proof}
Let $S = \con(x_1, \ldots, x_n)$.
Suppose $x_0 \in \pp(\tla)$, or equivalently (Proposition~\ref{prop:dist}),
$\|y - x_0\| \leq \|y - \tla\|$ for any $y \in T$.
Only the case of $x_0 \notin S$ needs to be studied.
Pick $x \in S$, and let $x^\prime$ be the point of $[x_0, x]$
with the property that
$d(x_0, x^\prime) = \inf \{ d(x_0, y) : y \in [x_0,x] \cap S\}$.
In other words, $x^\prime$ is the point of $[x_0, x] \cap S$,
farthest from $x$.
The set $S$ is closed, hence $x^\prime \in S$.

We claim that, for any $u \in S$, $x^\prime \in [x_0, u]$, and
consequently, $d(x_0,u) = d(x_0, x^\prime) + d(x^\prime, u)$.
Indeed, suppose, for the sake of contradiction, $x^\prime \notin [x_0, u]$.
Then there exists $z \in [x_0, x^\prime] \backslash \{x^\prime\}$ such that
$[x^\prime, u] = [x^\prime, z] \cup [z, u]$. By convexity, $z \in S$,
which is impossible, by the definition of $x^\prime$.

Next, we show that $x^\prime \in \pp(\tla)$. By Proposition~\ref{prop:dist},
it suffices to show that, for any $y \in T$,
$\|y - x^\prime\| \leq \|y - \tla\|$. We consider two cases:

(1) $[x^\prime, y] \cap S$ is strictly larger than $\{x^\prime\}$.
As $S$ is closed and metric convex, $[x^\prime, y] \cap S = [x^\prime, z]$,
for some $z$. We know that $[x_0, z] = [x_0, x^\prime] \cup [x^\prime, z]$,
hence $[x_0, y] = [x_0, x^\prime] \cup [x^\prime, y]$. Then
$d(x^\prime, y) = d(x_0, y) - d(x_0, x^\prime)$, and therefore,
$$
d(x^\prime, y) \leq d(x_0, y) \leq \|\tla - y\|
.$$

(2) $[x^\prime, y] \cap S = \{x^\prime\}$. In this case, note first
that, for any $u \in S$, $x^\prime \in [y, u]$, and
consequently, $d(y,u) = d(y, x^\prime) + d(x^\prime, u)$.
Indeed, if $x^\prime \notin [y, u]$, then there exists
$z \in [y, x^\prime] \backslash \{x^\prime\}$ such that
$[x^\prime, u] = [x^\prime, z] \cup [z, u]$. Then $z \in S$,
which contradicts our assumptions about $y$.

Now recall that the ambient space is $L_1(\Omega, \mu)$.
We can assume that $x^\prime = 0$. Then, for any $u \in S$,
$\|y - u\| = \|y\| + \|u\|$, hence $y u \leq 0$ $\mu$-a.e.~(we
view $y$ and $u$ as functions on $\Omega$).
As $x_1, \ldots, x_n \in S$, we also have $\tla y \leq 0$
$\mu$-a.e.. Therefore,
$\|y - \tla\| \geq \|y\| = d(x^\prime,y)$,
which is what we need.
\end{proof}

\begin{example}\label{l1_only}
Proposition ~\ref{prop:nearest} doesn't hold for embeddings into arbitrary spaces.
 Consider the ``spider''
$T = \{(i,t) : i \in \{1,2,3,4\}, 0 \leq t \leq 1\}$,
as in Example~\ref{big_set}. Embed $T$ into $\ell_\infty^3$ by setting
$(1,t) \mapsto (- e_1 + e_2 + e_3)t$, $(2,t) \mapsto (e_1 - e_2 + e_3)t$,
$(3,t) \mapsto (e_1 + e_2 - e_3)t$, and $(4,t) \mapsto (e_1 + e_2 + e_3)t$,
($e_1, e_2, e_3$ denote the canonical basis in $\ell_\infty^3$).
For $i = 1,2,3$, let $x_i = (i,1)$. Then the ``linear'' barycenter of
$\{x_1, x_2, x_3\}$ is $\tlx = (e_1 + e_2 + e_3)/3$. As this point
lies on the image of $T$ in $\ell_\infty^3$, $\pp(\tlx) = \{(4, 1/3)\}$.
This example also shows that the converse to
Corollary~\ref{cor:dist_tree} doesn't hold.
\end{example}

Finally, we present an example suggesting that nothing non-trivial can be said about
the distance from the ``linear'' barycenter of a tree to the tree itself.

\begin{example}\label{dist_to_tree}
Consider a ``spider'' $T$ with $n$ limbs of length $1$, that is, the set of
points $(i,t)$, with $1 \leq i \leq n$ and $0 \leq t \leq 1$, with the usual
radial metric. For $1 \leq i \leq n$ let $x_i = (i,1)$. Then there exists an
embedding of $T$ into $L_1(1,n+1)$ such that $\|x - \tlx\| \geq 1$ for any $x \in T$.
Indeed, the embedding taking $(i,t)$ to $\chi_{(i,i+t)}$ has the desired properties.
\end{example}

\section{Type, cotype, and convexity of metric trees}\label{type_cotype}

In this section we consider properties of metric spaces, such
as the four-point inequality and Reshetnyak's inequality (see
Definition~\ref{D:4pts}), type, and cotype.
The notion of metric type was introduced in \cite{BMW} (see also \cite{Pi}).
More recently, metric cotype was defined in \cite{MN}.


\begin{lemma}\label{l:4pts}
The four-point inequality implies Reshetnyak's inequality.
\end{lemma}
\begin{proof}
Suppose the elements $x_1, x_2, x_3, x_4$ of a metric space $(X,d)$ satisfy
$$
d(x_1,x_2) + d(x_3,x_4) \leq \max\{ d(x_1,x_3) + d(x_2,x_4) ,
d(x_1,x_4) + d(x_2,x_3) \} ,
$$
and show that
$$
d(x_1,x_2)^2 + d(x_3,x_4)^2 \leq d(x_1,x_3)^2 + d(x_2,x_4)^2 +
d(x_1,x_4)^2 + d(x_2,x_3)^2 .
$$
By scaling and relabeling, we can assume that
$$d(x_1,x_2) + d(x_3,x_4) = 1 \leq d(x_1,x_3) + d(x_2,x_4).$$
Let $a = d(x_1,x_2)$, $b = d(x_1,x_3)$. Then $d(x_3,x_4) = 1 - a$,
$d(x_2,x_4) = 1 - b$, and furthermore,
$$
\eqalign{
d(x_1,x_4)
&
\geq
|d(x_1,x_3) - d(x_3,x_4)| = |a+b-1| ,
\, \, \, {\mathrm{and}} \, \, \,
\cr
d(x_2,x_3)
&
\geq
|d(x_1,x_2) - d(x_1,x_3)| = |a-b| .
}
$$
Thus, it suffices to show that, for any $a \in [0,1]$ and $b \geq 0$,
$$
a^2 + (1-a)^2 \leq b^2 + (1-b)^2 + (a+b-1)^2 + (a-b)^2 .
$$
The last inequality is easily verified.
\end{proof}

Therefore, any metric tree is a CAT(0) space.
Below we show that metric trees are ``more convex'' (that is, their
moduli of convexity are larger) than those of ``generic'' CAT(0) spaces.

\begin{lemma}\label{l:convex}
Suppose $T$ is a complete metric tree. Then, for any $R > 0$ and $\vr \in [0,2R]$,
$\moco_M(R,\vr) \geq 1 - \vr/2$.
\end{lemma}

\begin{proof}
Suppose $a, x_1, x_2 \in T$ are such that $\max\{d(a,x_1), d(a,x_2)\} \leq R$,
and $d(x_1,x_2) \geq R \vr$. We have to show that $d(a,m) \leq R - R\vr/2$, where
$m$ is the midpoint of $[x_1,x_2]$. Find $y \in [x_1,x_2]$ such that
$[a,x_1] = [a,y] \cup [y,x_1]$, and $[a,x_2] = [a,y] \cup [y,x_2]$.
Relabeling if necessary, we can assume that $y \in [x_2,m]$.
Then $m \in [y,x_1]$, hence
$R \geq d(a,x_1) = d(a,m) + d(m,x_1) \geq d(a,m) + R \vr/2$.
Thus, $d(a,m) \leq R - R\vr/2$.
\end{proof}

\begin{definition}\label{D:metr_type}
Suppose $1 \leq p \leq 2$, and $K > 0$.
A metric space $(X,d)$ is said to have {\it metric type $p$} (or {\it BMW type $p$}),
after Bourgain, Milman, and Wolfson, who introduced this notion in \cite{BMW})
with constant $K$ if, for any $n \in \N$, and any function
$f : \{-1,1\}^n \to X$,
we have
\begin{equation}
\sum_{\vr \in \{-1,1\}^n} d(f(\vr), f(-\vr))^2 \leq
K n^{1/p-1/2} \sum_{\vr \in \{-1,1\}^n} \sum_{i=1}^n d(f(\vr), f(\vr^{[i]}))^2 ,
\label{def_type}
\end{equation}
where $(\vr_1, \ldots, \vr_n)^{[i]} =
(\vr, \ldots, \vr_{i-1}, - \vr_i, \vr_{i+1}, \ldots, \vr_n)$.
On an intuitive level, we can think of the points $f(\vr)$ as
vertices of a ``cube.'' Then the left hand side of (\ref{def_type}) is the sum
of the squares of the ``diagonals'' of this cube, while the right hand side
involves its ``edges.''

We do not quote the definition of metric cotype, due to space constraints.
Instead, we refer the reader to \cite{MN}.

\end{definition}


\begin{theorem}\label{thm:type}

\begin{enumerate}
\item
Any metric space satisfying the four-point inequality has metric type 2, with
constant $1$. In particular, this result holds for metric trees.
\item
Any complete metric tree has metric cotype 2, with a universal constant.
\end{enumerate}
\end{theorem}

\begin{proof}
Part (1) was proved in \cite{NS}.
For Part (2), recall that any $L_1$ space has cotype $2$,
with the constant $\sqrt{2}$ (this classical fact can be seen,
for instance, by combining the Khintchine constant from \cite{Sz}
with the basic properties of cotype, described in e.g.~\cite{JL}).
Therefore, by Theorem 1.2 of \cite{MN},
$L_1$ has metric cotype $2$, with the constant $90 \sqrt{2}$.
We have seen that any finitely
generated metric tree embeds isometrically into $\ell_1^N$, for some $N$. As the
cotype passes to subspaces, any finitely generated tree must have metric cotype $2$, with constant $90 \sqrt{2}$. Finally, metric cotype is a ``local'' property, hence any complete metric tree must possess it.
\end{proof}


We next tackle the negative type of metric trees.
Recall that a metric space $X$ has {\it{negative type $p$}}
($p > 0$) if,
for any $x_1, \ldots, x_n \in X$, the $n \times n$ matrix
$(d(x_i, x_j)^p)$ is conditionally negative definite.
Recall that a Hermitian matrix $A = (a_{ij})_{i,j=1}^n)$ is
conditionally negative definite if
$\sum_{i,j=1}^n a_{ij} \xi_i \overline{\xi_j} \leq 0$ whenever
the vector $\xi = (\xi_i)_{i=1}^n$ satisfies $\sum_i \xi_i = 0$.
The notion of $p$-negative type is equivalent to $p$-roundness,
see e.g. \cite{DW, LTW}.
Negative type is strongly related to positive definiteness of kernels,
and to embeddability into $L_p$-spaces (see e.g.~Section 8.1 of \cite{BL}).

It was shown in \cite{HLMT} that any metric tree has negative
type $1$. Therefore, it has negative type $p$ for any $p \in (0,1]$.
We shall show that a metric tree need not have negative type $p$ for
$p > 1$. More precisely, consider the ``spider'' $T_n$, consisting of
a central point, and $n$ limbs of length $1$.

\begin{proposition}\label{prop:neg_type}
If $p > 1$, then $T_n$ fails to have negative type $p$ for
$n$ large enough.
\end{proposition}

\begin{proof}
Suppose 
$n > c/(c-2)$, where $c = 2^p$.
Consider the subset of $T_n$, consisting of
the central point $t_0$, and the endpoints $t_1, \ldots, t_n$.
The corresponding $(n+1) \times (n+1)$ distance matrix is
$$
C = \left( \begin{array}{llllll}
   0      &  1     &  1     &  1     &  \ldots  &  1     \cr
   1      &  0     &  c     &  c     &  \ldots  &  c     \cr
   1      &  c     &  0     &  c     &  \ldots  &  c     \cr
   \ldots & \ldots & \ldots & \ldots &  \ldots  & \ldots \cr
   1      &  c     &  c     &  c     &  \ldots  &  0
\end{array} \right) .
$$
We shall show the existence of
$\xi = (\xi_0, \xi_1, \ldots, \xi_n)$ such that
$\xi_1 + \ldots + \xi_n = - \xi_0$, and
$\langle C \xi, \xi \rangle > 0$. Note that, for $\xi$ as above,
$D \xi = 0$, where the $D$ is a matrix of all whose entries equal $1$.
Let
$$
A = - \frac{1}{c-1} (C - cD) =
\left( \begin{array}{llllll}
   a      &  1     &  1     &  1     &  \ldots  &  1     \cr
   1      &  a     &  0     &  0     &  \ldots  &  0     \cr
   1      &  0     &  a     &  0     &  \ldots  &  0     \cr
   \ldots & \ldots & \ldots & \ldots &  \ldots  & \ldots \cr
   1      &  0     &  0     &  0     &  \ldots  &  a
\end{array} \right) ,
$$
where $a = c/(c-1) < 2$. It suffices to find
$\xi = (\xi_0, \xi_1, \ldots, \xi_n)$ such that
$\xi_1 + \ldots + \xi_n = - \xi_0$, and
$\langle A \xi, \xi \rangle < 0$.

An induction argument yields the determinant of this
$(n+1) \times (n+1)$ matrix: $\det A = a^{n+1} - n a^{n-1}$.
Thus, $A$ has $n-1$ eigenvalues equal to $0$, as well as non-zero
eigenvalues $\lambda_1 = a - \sqrt{n}$ and $\lambda_2 = a + \sqrt{n}$.
The corresponding normalized eigenvectors are
$\eta^1 = (-\sqrt{n}, 1, \ldots, 1)/\sqrt{2n}$ and
$\eta^2 = (\sqrt{n}, 1, \ldots, 1)/\sqrt{2n}$. Then,
for any $\xi \in \ell_2^n$,
$$
\langle A \xi, \xi \rangle =
\lambda_1 | \langle \xi, \eta^1 \rangle |^2 +
\lambda_2 | \langle \xi, \eta^2 \rangle |^2 .
$$
Now consider $\eta = (\sqrt{2n}, -\sqrt{2n}, 0, \ldots, 0)$. Then
$\langle \xi, \eta^1 \rangle = - \sqrt{n} - 1$, and
$\langle \xi, \eta^1 \rangle = \sqrt{n} - 1$. Therefore,
$$
\langle A \xi, \xi \rangle =
(a - \sqrt{n}) (\sqrt{n} + 1)^2 + (a + \sqrt{n}) (\sqrt{n} - 1)^2 =
2 a(n+1) - 4n ,
$$
which is negative, by our choice of $n$.
\end{proof}

Finally, we note that all metric trees have Markov type $2$
\cite{NPSS}.

\section{Entropy Quantities and Other Measures of Compactness}\label{compact}

\subsection{$\epsilon$-entropy and related quantities}
Kolmogorov introduced the notion of $\epsilon$-entropy as
a measure of the massiveness of sets \cite{Kolm}.
    This notion has been useful in function spaces
    (see \cite{Edmu}), especially with asymptotic distribution of
    eigenvalues of elliptic operators, or as a way of measuring the
    sizes of spaces of solutions to PDE's \cite{Chep}. Recently
    entropy and $n$-widths has been utilized as a measure of efficiency in the task of data
    compression (see \cite{Dovd}, \cite{Posn}, \cite{Dono}).
   In this section we examine  the notion of entropy and its connection to the fact that
complete metric trees are centered (Theorem~\ref{T:cmt centered}).
This proves the useful fact that the $\epsilon$-entropy of a bounded
subset $A$ of a tree $T$ is equal to the $\epsilon$-entropy of $A$ relative to
$(T,d)$. We also connect the covering numbers of a compact subset of a tree with these of its
convex hull (see e.g.~\cite{CKP} for some Banach space results in the same vein).

\begin{definition}\label{D:netdist}
Suppose $A$ is a subset of a metric space $M$.
\begin{itemize}
\item
$A$ is \emph{centered} if for all $U\subset A$ such that $\diam(U) =2r$, there exists
$a\in A$ such that $U \subset B_{c} (a;r)$. By $B_{c} (a;r)$ we mean the closed ball
of radius $r$ centered at $a$.
 \item
 $\{m_i\}_{i\in I} \subset M$
      is an \emph{$\ep$-net for $A$ in $M$} if

    $$A \subset \bigcup_{i \in I} B(m_i;\ep)$$

   \item   $\{U_\alpha\}_{\alpha \in I}$,
    where $U_\alpha \subset M $, is an \emph{$\ep$-cover for $A$} if
    $\diam(U_\alpha) \leq 2\ep$ and $$A \subset \displaystyle \bigcup_{\alpha \in
    I}U_\alpha .$$
     \item
     $U \subset A$ is a \emph{$\ep$-separated subset of $A$} if

   $$ \ep \leq x_ix_j \quad \text{for all }i,j \in I \text{ with } i \not=j,\,\, x_i \not=x_j \in U.$$

\end{itemize}
Let $\mathcal{N}_\ep(A)$ ($\mathcal{K}_\ep^{M}(A)$) be the cardinality of a minimal
$\ep$-cover of $A$ (respectively ~minimal $\ep$-net for $A$ in $M$). 
Define $\mathcal{M}_{\ep}(A)$ as the maximal cardinality of an $\ep$-separated subset
of $A$.
\end{definition}

Note that, if $A$ is a complete metric tree, then it is injective,
hence $\mathcal{K}_\ep^{M}(A) = \mathcal{K}_\ep^{A}(A)$ for any ambient space $M$.

\begin{theorem}\label{T:cmt centered}
    Every complete metric tree $T$ is centered.
\end{theorem}

Very few spaces are centered. A typical example of
a space which is not centered is $\R^2$. This can be seen if one
tries to locate a center for an equilateral triangle of side length
$2r$ so that its distance to all points is at most $r$.

For the proof we require a lemma.

\begin{lemma}\label{L:bcinmt}
    Let $A$ be a subset of metric tree $T$ with $\diam(A) = 2r$.
    Then for all $\ep >0$ exists $m \in con(A) $ such that $A \subset B(m;r+\ep)$.
\end{lemma}
\begin{proof}
    For all $\ep >0$, there exists $x,y \in A$ such that $xy > 2r - 2\ep$ and let
    $m$ be the midpoint of $[x,y]$.

    Let $z \in A$, then by the three-point property of metric trees,
    there exists $w \in [x,y]$ such that $[z,x] \cap [z,y] = [z,w]$.  Without loss
    of generality we can assume that $m \in [w,x]$ and hence $w \in [z,x]$
    by transitivity.  Next $\diam(A) = 2r$ and $w \in [z,x]$ imply that
    \[
    2r \geq zx = zm + mx = zm + \frac{xy}{2} > zm + (r-\ep),
    \]
    which implies $r + \ep > zm$.  Therefore, $A \subset B(m; r+\ep)$.
\end{proof}

\begin{proof}[Proof of \ref{T:cmt centered}]
    Let $U$ be a bounded subset of a metric tree $T$, and let $\diam(U)=2r$.

    For all $n \in \N$, there exists $x_n,y_n \in U$ such that $x_ny_n > 2(r-n^{-1})$.
    Let $z_n \in T$ be the midpoint of $[x_n,y_n]$, and we claim $\{z_n\}$ is a Cauchy sequence.

    Let $0<2N^{-1}< \ep$ with $N \in \N$, and then let $n,m \geq N$.
    Let $u \in [x_n,y_n]$ be such that $[z_m,u] = [z_m,x_n] \cap [z_m,y_n]$,
    by swapping $x_n$ and $y_n$ we can claim without loss of generality that
    $u \in [z_n,y_n]$.  Therefore, $z_n \in [x_n,z_m]$.

    Since metric segments are closed under intersections, $z_m \in [x_m,y_m]$,
    and $z_m$ is an end point of $[x_n, z_m]$, we have that
    $[x_n, z_m] \cap [x_m, y_m] = [z_m, v]$ where $v \in [x_m, y_m]$.
    Hence, by switching $x_m$ and $y_m$ we can claim without loss of generality that
    $v \in [x_m, z_m]$.  Therefore $z_m \in [x_n,y_m]$.

    Since $\diam(U) = 2r$, $x_n,y_m \in U$, $z_n \in [x_n,z_m]$
    and $z_m \in [x_n,y_m]$, we have
    \[\eqalign{
    2r
    & \geq x_ny_m = x_nz_n + z_nz_m + z_my_m
    \cr & > (r-n^{-1}) + z_nz_m + (r-m^{-1})
    \geq 2r - 2N^{-1} +z_nz_m.
    }\]
    Therefore, $z_nz_m < 2N^{-1} < \ep$ and $\{z_n\}$ is Cauchy.
    $M$ is complete so let $\lim_n z_n = z$.

    Suppose that there exists $u \in U$ such that $zu > r +2\ep$ for some $\ep >0$.
    Since $\lim_n z_n = z$, we can find a $n$ such that $n^{-1} < \ep$ and
    $zz_n < \ep$.  Furthermore, by the proof of Lemma~\ref{L:bcinmt} we know that
    $z_nu < r + n^{-1} < r+ \ep$.  Hence, by the triangle inequality,
    $zu \leq zz_n + z_nu < r + 2\ep$, which contradicts that $zu > r +2\ep$.
    Hence, $U \in B_c(z;r)$ and therefore, $T$ is centered.
\end{proof}
\begin{remark}
Alternatively, one can prove Theorem \ref{T:cmt centered} by recalling, from
Theorem~\ref{T:cmt is hc}, that any complete metric tree is hyperconvex.
By the definition of hyperconvexity (Definition~\ref{D:hc}), any hyperconvex
set is centered. However, our proof relies only on the properties of the metric segments,
and thus sheds more light on the local property of trees.
\end{remark}

\begin{theorem}
If $A$ is a subset of a complete metric tree $T$, then
$\mathcal{K}_\ep^{T}(A) = \mathcal{M}_\ep(A)$.
\end{theorem}

\begin{proof}
Any complete metric tree is centered.
Thus, any $\epsilon$-net for $A$ is equivalent to an $\epsilon$-cover.
\end{proof}

Next we connect the covering numbers $\mathcal{N}_\ep(S)$ of a compact subset $S$ of a tree $T$ with those of its
convex hull (see e.g.~\cite{CKP} for some Banach space results).

\begin{theorem}\label{thm:conv}
Suppose $S$ is a compact subset of a complete metric tree $T$ and $\vr_1, \vr_2$
are positive numbers. Then
$$\mathcal{N}_{\vr_1+\vr_2}(\con(S)) \leq \mathcal{N}_{\vr_1}(S) \lceil \diam S/(4\vr_2) \rceil.$$
\end{theorem}

\begin{proof}
For the sake of brevity, set $N = N_{\vr_1}(S)$, $D = \diam(S)$, and $S^\prime = \con(S)$.
Convexity of the norm (see \cite{St}) implies that the diameter of $S^\prime$ equals $D$. By Theorem ~\ref{T:cmt centered}, there exists $x_0 \in S^\prime$ such that for any $y \in S^\prime$,
$d(x_0,y) \leq \diam(S^\prime)/2 = D/2$, and moreover,
$S^\prime = \displaystyle \bigcup_{x \in S} [x_0,x]$.

Find $x_1, \ldots, x_N \in T$ such that for any $x \in S$ there exists $i$
with the property that $d(x,x_i) \leq \vr_1$. Let $x_i^\prime$ be the point of
$S^\prime$ which is closest of $x_i$. Then $d(x_i^\prime, y) \leq d(x_i, y)$
for any $y \in S^\prime$. Indeed, there exists $z \in [x_i^\prime, y]$ such that
$[x_i, x_i^\prime] = [x_i, z] \cup [z, x_i^\prime]$. By convexity, $z \in S^\prime$,
hence $z = x_i^\prime$, which is what we need.

Now let $K = \lceil D/(4\vr_2) \rceil$. For each $i$, find  the points
$(y_{ij})_{j=1}^K$ on $[x_0,x_i^\prime]$ in such a way that
$d(x_0,y_{i1}) \leq \vr_2$, $d(x_i^\prime,y_{iK}) \leq \vr_2$, and
$d(y_{ij},y_{i,j+1}) \leq 2 \vr_2$ for $1 \leq j \leq K$.
In total, we have $NK$ points $y_{ij}$. It remains to show that, for any $y \in S$,
$d(y, y_{ij}) \leq \vr_1 + \vr_2$ for some $(i,j)$.

As we have observed, there exists $x \in S$ such that $y \in [x_0,x]$. Find $i$ such that
$d(x,x_i^\prime) \leq \vr_1$. By Corollary 2.5 of \cite{St}, there exists
$z \in [x_0,x_i^\prime]$ such that $d(z, y) \leq \vr_1$. Furthermore, there exists $j$
such that $d(z, y_{ij}) \leq \vr_2$. By the triangle inequality,\\
$d(y, y_{ij}) \leq \vr_1 + \vr_2$.
\end{proof}



\subsection{Kolmogorov numbers}
Kolmogorov introduced the notion of
diameters (or widths) to generalize many of our intuitive ideas about
``flatness'' of compact subsets of linear spaces.
Since then, Kolmogorov diameters have been widely used in approximation theory
(see \cite{Pink} and references therein). On the other hand, the notion of the measure of non-compactness of a subset of a metric
space was introduced by Kuratowski \cite{Kura} as a way to
generalize Cantor's intersection theorem.  In 1955, Darbo
\cite{Darb} applied measures of non-compactness to prove a powerful
fixed point theorem. Since then measures of non-compactness
have been a standard notion in fixed point theory. In the following,
we define these two concepts and show the connections between them.

\begin{definition}\label{D:kol diam nls}
    Given a subset $A$ of a normed linear space $X$ and $n \geq 0$, define the
    \emph{$n$-th Kolmogorov diameter ($n$-width) of $A$ in $X$} as:
    \[
    \delta_n(A, X) = \delta_n(A)
    := \inf \left\{ \sup_{a \in A} \, d(a, M) \mid M
        \text{ is a $n$-dimensional subspace of } X \right\}.
    \]
The \emph{$n$-th affine Kolmogorov diameter of $A$ in $X$}
is defined as:
\[
\delta_n^{(a)}(A, X) = \delta_n^{(a)}(A)
:=
    \inf \left\{ \sup_{a \in A} \, d(a, M)
    \mid M \text{ is an affine subspace of } X, \, \dim M \leq n \right\}.
\]
\end{definition}

Observe that the sequences $\{\delta_n(A)\}_{n=1}^{\infty}$ and
$\{\delta_n^{(a)}(A)\}_{n=1}^{\infty}$ are non-increasing, and
$$\delta_n(A) \geq \delta_n^{(a)}(A) \geq \delta_{n+1}({\mathrm{conv}} \,(A \cup (-A))).$$
Indeed, the left hand side inequality is obvious.
To establish the right hand side, suppose $\delta_n^{(a)}(A) < c$. Then there exists
an affine subspace $M$, of dimension not exceeding $n$, such that for any $a \in A$ there
exists $m \in M$ satisfying $\|a - m\| < c$. Any $x \in A \cup (-A)$ can be expressed as
$x = \sum_{i=1}^N \alpha_i a_i$, with $\sum_i |\alpha_i| \leq 1$, and $a_i \in A$.
For each $i$, find $m_i \in M$ such that $\|a_i - m_i\| < c$. Then
$M^\prime = {\mathrm{conv}} \,(M \cup (-M))$ is a linear subspace of dimension
not exceeding $n+1$, $m = \sum_i \alpha_i m_i \in M^\prime$, and
$\|x - m\| < c$.

Furthermore, if $A$ is centrally symmetric about $0$ in $X$, then
$\delta_n(A) = \delta_n^{(a)}(A)$. Indeed, fix $\vr > 0$, and find
an affine subspace $M \subset X$ of dimension not exceeding $n$, such that
for any $a \in A$ there exists $m \in M$ with the property that
$\|a-m\| < \delta_n^{(a)}(A) + \vr$. By symmetry, for such an
$a$ we can also find $m_- \in -M$ satisfying
$\|a-m_-\| < \delta_n^{(a)}(A) + \vr$. Note that
$m^\prime = (m + m_-)/2 \in M^\prime = M + (-M)$, and the latter
is a linear subspace of $X$, of the same dimension as $M$. By the
triangle inequality, $\|a - m^\prime\| < \delta_n^{(a)}(A) + \vr$.
Thus, $\delta_n(A) \leq \delta_n^{(a)}(A) + \vr$. As $\vr > 0$
is arbitrary, we are done.

\begin{definition}\label{D:non-compact}
Suppose $A$ is a subset of the metric space $M$. Define the
    \emph{ball (Hausdorff) measure of non-compactness} and the
\emph{set measure of non-compactness} as
    \[
    \beta (A,M):= \inf \big\{b>0 \mid A \subset \bigcup_{j=1}^{n} B(m_j;b)
    \text{ for some } m_j \in M  \big\}.
    \]
and
$$
\alpha(A):=\inf  \big\{a > 0 \mid A \subset \displaystyle\bigcup _{j=1}^{k} A_j
\,\,\,\mbox{for some}\,\,\, A_j \subset A\,\,\,\mbox{with}
\,\,\, \mbox{diam}(A_j) \leq a \big\} ,
$$
respectively.
\end{definition}

Note that $\beta(A,M)$ is an ``extrinsic'' measure of non-compactness, and
may depend on the ambient space $M$. On the other hand, $\alpha(A)$ is intrinsic,
and is independent of $M$. It is easy to observe that
$\beta(A,M) \leq \alpha(A) \leq 2 \beta(A,M)$.


Connections between entropy of linear maps, their Kolmogorov
numbers (and other $s$-numbers), and their analytic properties,
such as eigenvalues and essential spectrum, have been studied extensively
(see \cite{Carl} and references therein).
Below we present some results illuminating the connections between
Kolmogorov numbers and entropy properties of metric spaces.


\begin{theorem}\label{T: lim kol diam to b}
Suppose $A$ is a bounded subset of a Banach space $X$. Then
$$\lim_{n \to \infty} \delta_n(A,X) = \beta(A,X) = \lim_{n \to \infty} \delta_n^{(a)}(A,X).$$
\end{theorem}

\begin{proof}
We only show the equality involving $\delta_n(A,X) = \delta_n(A,X)$, as the one with
$\delta_n^{(a)}(A,X)$ is handled in a similar manner.
By the boundedness of $A$, $\{\delta_n(A,X)\}_{n=1}^{\infty}$ forms a non-increasing
sequence of nonnegative numbers, hence $\lim_{n \to \infty} \delta_n(A,X)$
exists.

(1) $\beta(A,X) \leq \lim_{n \to \infty} \delta_n(A,X)$.
Pick $c > b > \lim_n \delta_n(A,X)$, and show $\beta(A,X) \leq c$.
Thus, there exists $n$ such that $\delta_n(A,X) < b$.
This means there there exists an $n$-dimensional subspace
$E$ of $X$ such that $\sup_{a \in A} d(a, E) < b$.
Let $Q = \{e \in E : \|e\| \leq b + \sup_{a \in A} \|a\|\}$.
By compactness, $Q$ contains a finite $(c-b)$-net $\{q_n\}$.
Then $A \subset \cup_n B(q_n, c)$, hence $\beta(A,X) \leq c$.

(2) $\beta(A,X) \geq \lim_{n \to \infty} \delta_n(A,X)$.
    Let $b > \beta(A,X)$, and show that $\delta_n(A,X) \leq b$.
    Find a finite $b$-net $\{x_j\}_{j=1}^{n} \subset X$ for $A$.
    Then the dimension of $E = \vecspan(\{x_j\}_{j=1}^{n})$ does not exceed $n$.
    For any $a \in A$,
$\dist(a,E) \leq \min_j \|a - x_j\| < b$, hence $\delta_n(A,X) \leq b$.
\end{proof}

\begin{corollary}
Suppose a complete metric tree $T$ is embedded isometrically into
a Banach space $X$, and $A$ is a bounded subset of $T$. Then
$$
\lim_{n \to \infty} \delta_n(A,X) = \lim_{n \to \infty} \delta_n^{(a)}(A,X) =
\frac{\alpha(A)}{2} .
$$
\end{corollary}

\begin{proof}
Clearly, $\alpha(A) \leq 2 \beta(A,X)$. By Theorem~\ref{T:cmt centered},
$ \alpha (A)=2 \beta(A, T) \geq 2 \beta(A,X)$. Thus, $\alpha(A) = 2 \beta(A,X)$.
An application of Theorem~\ref{T: lim kol diam to b} completes the proof.
\end{proof}


Next we consider affine Kolmogorov diameters of $V(S)$, where $V$ is an
embedding of a metric space $S$ into a Banach space $X$.
It is well known (see e.g. \cite{Carl, Pink}, or Remark~\ref{rem:kolm_smallest} below)
that Kolmogorov diameters may depend heavily on the ambient space $X$.
If $X$ is a subspace of $Y$, then $\delta_n(V(S),Y) \leq \delta_n(V(S),X)$.
Furthermore, if $X$ is contained in a $\lambda$-injective space $Z$, then
$\delta_n(V(S),Z) \leq \lambda \delta_n(V(S),Y)$.
Similar inequalities hold for $\delta_n^{(a)}$.

Suppose now that $V$ is an isometric embedding of $S$ into a
$1$-injective Banach space $X$. Then, by Theorem~\ref{thm:universal},
$\delta_n^{(a)}(V(S),X) \leq d_n(S)$, where
$d_n(S) = \delta_n^{(a)}(U(S),\ell_\infty({\mathcal{L}}))$
(the universal embedding $U : S \to \ell_\infty({\mathcal{L}})$
was defined in Section~\ref{subs:into_l_infty}).

\begin{proposition}\label{prop:kol_numbers}
Suppose $S$ is a metric space.
Let $$c_1 = \inf \{\ep > 0 \mid \mathcal{K}_\ep^S(S) \leq n\},\,\,\,\mbox{ and}\,\,\,
c_2 = \sup \{\ep > 0 \mid \mathcal{M}_\ep(S) \geq n+1\}.$$
\begin{enumerate}
\item
Suppose $X$ is a Banach space, and $V : S \to X$ is a $1$-Lipschitz map.
Then $\delta_n^{(a)}(V(S)), X) \leq c_1$. Thus, $d_n(S) \leq c_1$.
\item
$d_n(S) \geq c_2/2$.
\end{enumerate}
\end{proposition}

\begin{proof}
(1) Fix $c > c_1$, and suppose $s_1, \ldots, s_n$ is a $c$-net in $S$.
Let $E$ be the affine span of
$V(s_1), \ldots, V(s_n)$. Then $\dim E \leq n$.
Furthermore, for any $s \in S$,
$$
d(V(s), E) \leq \min_i d(V(s), V(s_i)) \leq \min_i d(s, s_i) \leq c ,
$$
which shows that $\delta_n^{(a)}(V(S)), X) \leq c$.
As $c > c_1$ is arbitrary,
$\delta_n^{(a)}(V(S)), X) \leq c_1$.

(2) Let $M = n+1$. For $c < c_2$, let
$s_1, \ldots, s_M$ be a $c$-separated subset of $S$.
Then there exists $s_0 \in T$ such that
$d(s_0, s_i) > c/2$ for any $i$. Indeed, by relabeling if necessary,
we can assume that $d(s_1,s_2) \leq d(s_i,s_j)$ whenever
$i$ and $j$ are different. Let $s_0$ be the midpoint of
$[s_1, s_2]$. We claim that $d(s_0,s_i) > c/2$ for any $i$. The
inequality clearly holds for $i \in \{1,2\}$. If $i > 2$ and
$d(s_0,s_i) \leq c/2$, then
$$
d(s_1, s_i) \leq d(s_1, s_0) + d(s_0, s_i) \leq
\frac{d(s_1, s_2)}{2} + \frac{c}{2} < d(s_1, s_2) ,
$$
a contradiction.

Let ${\mathcal{L}}$ be the set of $1$-Lipschitz functions from $T$
to $\R$, taking $s_0$ to $0$. For any
$\sigma = (\sigma_1, \ldots, \sigma_M) \in \{-1,1\}^M$, define
$g_\sigma : \{s_0, s_1, \ldots, s_M\} \to \{-c/2, c/2\} \subset \R$
by setting $g_\sigma(s_0) = 0$, and $g_\sigma(s_i) = \sigma_i c/2$ for
$i \geq 1$. Clearly $g_\sigma$ is $1$-Lipschitz. By the injectivity
of $\R$, it has an extension $h_\sigma : S \to \R$,
belonging to ${\mathcal{L}}$.

Let $L = 2^M$, and identify $\{-1,1\}^M$ with $\{1, \ldots, L\}$.
Define the map $A : \ell_\infty({\mathcal{L}}) \to \ell_\infty^L :
(a_h)_{h \in {\mathcal{L}}} \to (a_{h_\sigma})_{\sigma=1}^L$.
Clearly, this is a linear contraction, hence $d_n(S) \geq
\delta_n^{(a)}(A \circ U(S))$. Moreover, $U(s_i) = (h(s_i))_{h \in {\mathcal{L}}}$.
For $1 \leq i \leq M$ consider
$e_i = A \circ U(s_i) = (h_\sigma(s_i))_\sigma \in \ell_\infty^L$.
For any real numbers $\alpha_1, \ldots, \alpha_M$,
$$
\|\sum_i \alpha_i e_i\| = \max_\sigma |\sum_i \alpha_i h_\sigma(s_i)| =
\frac{c}{2} \sum_i |\alpha_i| .
$$
Therefore, $e_1, \ldots, e_M$ are linearly independent.
Moreover, for $x \in \span[\pm e_1, \ldots, \pm e_M]$,
$\|x\| \leq c/2$ if and only if $x \in C$, where
$C = {\mathrm{conv}} (\pm e_1, \ldots, \pm e_M)$.
In other words, $C$ is the ball of a $M$-dimensional
subspace of $\ell_\infty^L$, of radius $c/2$.
By Lemma 2.c.8 of \cite{LT1}, $\delta_k(C,\ell_\infty^L) = c/2$ if $k < M$.
The set $C$ is centrally symmetric, hence
$\delta_k(C,\ell_\infty^L) = \delta_k^{(a)}(C,\ell_\infty^L)$.
As $n = M-1$, we conclude
$$
d_n(S) \geq \delta_n^{(a)}(A \circ U(S)) \geq
\delta_n^{(a)}(C,\ell_\infty^L) \geq
\delta_n(C,\ell_\infty^L) = \frac{c}{2} .
$$
Since this inequality is valid for any $c < c_2$, we are done.
\end{proof}

As an application, we estimate $d_n(T)$ for finitely generated trees.
That is, suppose $T$ arises from a weighted graph theoretical tree
${\mathcal{T}}$ (see Example~\ref{E:graph}).
For such a tree, denote by $|T|$ the sum of weights (lengths)
of the edges of the original graph.

\begin{corollary}\label{cor:fin_gen_kolm}
There exist $0 < c_1 < c_2$ with the property that, for any
finitely generated tree $T$, there exists $N = N(T) \in \N$ such that
the inequality $c_1 |T|/n \leq \delta_n^{(a)}(T) \leq c_2 |T|/n$
holds for any $n \geq N = N(T)$.
\end{corollary}

In fact, one can see that $N(T)$ depends on the minimum
of lengths of the edges of ${\mathcal{T}}$.

\begin{remark}\label{rem:kolm_smallest}
For a metric tree $T$, we have no good estimates for
$\inf_{V,X} \delta_n^{(a)}(V(T), X)$, where the infimum runs over all
isometric embeddings $V$ of $T$ into a Banach space $X$.
As we are interested in the infimum, we can assume that $X = \ell_\infty(I)$,
for some index set $I$. For certain trees $T$ and isometric embeddings
$A$, $\delta_n^{(a)}(V(T), X)$ can be much smaller than $d_n(T)$.
For instance, pick $N \in \N$, and let $L = 2^N$. Consider a ``spider'' $T$
with $L$ limbs of length $1$. More precisely, $T$ consists of
the ``root'' $o$, and the pairs $(i,t)$, with $1 \leq i \leq L$, and
$0 < t \leq 1$. For convenience, we identify $o$ with $(i,0)$.
The metric on $T$ is described in Example~\ref{E:tripod}.
We can embed $T$ into $\ell_\infty^N$ isometrically: let
$e_1, \ldots, e_L$ be an enumeration of the vertices of the
unit ball of $\ell_\infty^N$. Then the map $V : T \to \ell_\infty^N$,
taking $(i,t)$ to $t e_i$, is an isometric embedding.
Therefore, $\delta_n^{(a)}(V(T),\ell_\infty^N) = 0$ for $n \geq N$.
On the other hand, $T$ contains a $2$-separated set of
cardinality $L = 2^N$ (the endpoints of the limbs of $T$).
By Theorem~\ref{prop:kol_numbers}, $d_n(T) \geq 1$ for $n \leq L-1$.
\end{remark}

\section*{Acknowledgments.}

Some of the results of this paper were presented at the 19$^\mathrm {th} $
International Conference on Banach Algebras held at B\c{e}dlewo, July 14--24,
2009. The authors wish to thank the Polish Academy of Sciences, The European
Science Foundation (ESF-EMS-ERCOM partnership) and the Faculty of Mathematics
and Computer Science of the Adam Mickiewicz University at Pozna\'n.
The authors are grateful to the referee for many useful remarks, which
helped to make the presentation more transparent.

\newpage

\bibliographystyle{amsplain}


\end{document}